\documentclass[12pt,a4paper]{amsart}
\usepackage{latexsym,amsfonts,amsthm,amsmath,mathrsfs,amssymb}
\usepackage[british]{babel}
\usepackage[dvips]{color}
\usepackage[utf8]{inputenc}
\usepackage[T1]{fontenc}
\usepackage[dvips,final]{graphics}
\usepackage{xcolor}
\usepackage{mathrsfs}
\usepackage{mathtools} 
\usepackage{breqn}
\usepackage{epstopdf}
\usepackage{verbatim}
\usepackage{amscd}
\usepackage{hyperref}
\usepackage[english,capitalize]{cleveref}
\usepackage{aliascnt}

\makeatletter
\def\newaliasedtheorem#1[#2]#3{
	\newaliascnt{#1@alt}{#2}
	\newtheorem{#1}[#1@alt]{#3}
	\expandafter\newcommand\csname #1@altname\endcsname{#3}
}
\makeatother

\theoremstyle{plain}

\newtheorem{theorem}{Theorem}[section]
\newaliasedtheorem{prop}[theorem]{Proposition}
\newaliasedtheorem{lem}[theorem]{Lemma}
\newaliasedtheorem{coroll}[theorem]{Corollary}

\theoremstyle{definition}
\newaliasedtheorem{defi}[theorem]{Definition}
\newaliasedtheorem{quest}[theorem]{Question}
\newaliasedtheorem{fact}[theorem]{Fact}

\theoremstyle{remark}
\newaliasedtheorem{rem}[theorem]{Remark}
\newaliasedtheorem{exa}[theorem]{Example}

\newcommand{\R}{\mathbb{R}}

\newcommand{\W}{\mathbb{W}}

\newcommand{\HH}{\mathbb{H}}

\newcommand{\G}{\mathbb{G}}

\newcommand{\sW}{\mathbb{W}}
\newcommand{\sL}{\mathbb{L}}

\newcommand{\eps}{\varepsilon}
\let\altphi\phi
\let\phi\varphi
\let\varphi\altphi
\let\altphi\undefined

\DeclareMathOperator{\de}{\mathrm d\!}

\newcommand{\negphantom}[1]{\settowidth{\dimen0}{#1}\hspace*{-\dimen0}}

\newcommand{\average}{{\mathchoice {\kern1ex\vcenter{\hrule height.4pt
width 6pt
depth0pt} \kern-9.7pt} {\kern1ex\vcenter{\hrule height.4pt width 4.3pt
depth0pt}
\kern-7pt} {} {} }}

\renewcommand{\epsilon}{\varepsilon}

\address{\textsc{Gioacchino Antonelli}: 
	Scuola Normale Superiore, Piazza dei Cavalieri, 7, 56126 Pisa, Italy.}
\email{gioacchino.antonelli@sns.it}
\address{\textsc{Daniela Di Donato}: 
Department of Mathematics and Statistics, P.O.\ Box 35 (MaD), FI--40014, University of Jyv\"askyl\"a, Finland.}
\email{daniela.d.didonato@jyu.fi}
\address{\textsc{Sebastiano Don}:
Department of Mathematics and Statistics, P.O.\ Box 35 (MaD), FI--40014, University of Jyv\"askyl\"a, Finland.}
\email{sedon@jyu.fi}

\thanks{
	D.D.D., S.D. are partially supported by the Academy of Finland (grant
	288501
	`\emph{Geometry of subRiemannian groups}' and by grant
	322898
	`\emph{Sub-Riemannian Geometry via Metric-geometry and Lie-group Theory}'). G.A., D.D.D., S.D. are partially supported by the European Research Council
	(ERC Starting Grant 713998 GeoMeG `\emph{Geometry of Metric Groups}').
}

\title{Distributional solutions of Burgers' type equations for intrinsic graphs in Carnot groups of step 2}

\date{\today}

\author{Gioacchino Antonelli, Daniela Di Donato and Sebastiano Don}

\begin{document}
	\begin{abstract}
		We prove that in arbitrary Carnot groups $\mathbb G$ of step 2, with a splitting $\mathbb G=\mathbb W\cdot\mathbb L$ with $\mathbb L$ one-dimensional, the graph of a continuous function $\phi\colon U\subseteq \mathbb W\to \mathbb L$ is $C^1_{\mathrm{H}}$-regular precisely when $\phi$ satisfies, in the distributional sense, a Burgers' type system $D^{\phi}\phi=\omega$, with a continuous $\omega$. We stress that this equivalence does not hold already in the easiest step-3 Carnot group, namely the Engel group. 
		
		As a tool for the proof we show that a continuous distributional solution $\phi$ to a Burgers' type system $D^{\phi}\phi=\omega$, with $\omega$ continuous, is actually a broad solution to $D^{\phi}\phi=\omega$. As a by-product of independent interest we obtain that all the continuous distributional solutions to $D^{\phi}\phi=\omega$, with $\omega$ continuous, enjoy $1/2$-little H\"older regularity along vertical directions. 
	\end{abstract}
	\maketitle 
	\tableofcontents

\section{Introduction}

Due to the multitude of applications, sub-Riemannian geometry has attracted a lot of attention in the mathematical community in the recent years.
A sub-Riemannian manifold is a generalization of Riemannian manifold for which the metric is induced by a smooth scalar product only defined on a sub-bundle of the tangent bundle. The infinitesimal model of a sub-Riemannian manifold, namely the class of its Gromov-Hausdorff tangents, is represented by the class of (quotients of) Carnot groups \cite{SC16, LeDonne}. Carnot groups are connected and simply connected Lie groups $\G$ whose Lie algebra $\mathfrak g$ admits a stratification, namely a decomposition into nontrivial complementary linear subspaces $V_1,\dots, V_s$ such that
\[
\mathfrak g=V_1\oplus\dots\oplus V_s,\quad [V_j,V_1]=V_{j+1},\quad \text{for } j=1,\dots, s-1,\quad [V_s,V_1]=\{0\},
\] 
where $[V_j,V_1]$ denotes the subspace of ${\mathfrak{g}}$ generated by
the commutators $[X,Y]$ with $X\in V_j$ and $Y\in V_1$.
Carnot groups have been studied from very different point of views such as Differential Geometry \cite{CDPT}, 
Subelliptic Differential Equations \cite{BLU, Folland, Folland2, Sanchez-Calle}, Complex Analysis \cite{stein} and Neuroimaging \cite{CitManSar}.
 
Concerning Geometric Measure Theory in the setting of Carnot groups, one of the most studied problems in the past twenty years is represented by the {rectifiability problem}: is it possible to cover the boundary of a finite perimeter set with a countable union of $C^1$-regular surfaces? The answer to this question is affirmative in the Euclidean case and it was studied in \cite{DeGiorgi54, DeGiorgi55} via a blow-up analysis. The proof of De Giorgi has then been adapted in the framework of step-2 Carnot groups in \cite{FSSCHeis,FSSC03} and then generalized to the so-called Carnot groups of type $\star$ in \cite{Marchi}, see also the recent \cite{LDM}.
When dealing with Carnot groups of step 3 or higher, only partial results concerning this question are available in the literature. One of the main difficulty is represented by the fact that it is not known in general if $C^1$ rectifiability is equivalent to a Lipschitz-type rectifiability. Concerning Heisenberg groups, see \cite{Vittone20} for a Rademacher-type theorem for intrinsic Lipschitz graphs of any codimension. Different notions of rectifiability have also been recently investigated, see \cite{ALD, DLDMV}.

The rectifiability problem represents an example that underlines the importance of a fine understanding of intrinsic surfaces inside Carnot groups. The study of different notions of surfaces in Carnot groups has been quite extensive in the recent years and we mention \cite{FSSC07} for a definition of regular submanifold in the Heisenberg groups, \cite{FMS14, FS16} for intrinsic Lipschitz graphs and their connection to $C^1$-hypersurfaces, \cite{Magnani18} for a notion of non-horizontal transversal submanifold and \cite{Mag13, JNGV} for a notion of $C^1$-surface with Carnot group target, but the list is far from being complete.

We focus our attention on codimension-one intrinsic graphs. A codimension-one intrinsic graph $\Gamma$ inside a Carnot group $\G$ comes with a couple of homogeneous and complementary subgroups $\sW$ and $\sL$ with $\mathbb L$ one-dimensional, see \cref{sec:preliminari}, and a map $\phi\colon U\subseteq \sW\to \sL$ such that $\Gamma=\{x\in \G: x=w\cdot\phi(w), w\in U \}$.
 It turns out that the regularity of the graph $\Gamma$ is strictly related to the regularity of $\phi$ and its {intrinsic gradient} $\nabla^\phi\phi$, see \cref{sec:preliminari}. 
 As a geometric pointwise approach, we just say that $\phi$ is intrinsically differentiable if its graph has a homogeneous subgroup as blow-up. However, one can define some different notions of regularity that rely on some
  $\phi$-dependent operators $D^\phi_{W}$ whenever $W\in {\rm Lie}(\sW)$, see \cref{def:PhiJ}. If an adapted basis of the Lie algebra $(X_1,\dots, X_n)$ is fixed and is such that  $\sL\coloneqq \exp({\rm span}\{X_1\})$ and $\sW\coloneqq \exp({\rm span}\{X_2,\dots, X_n\})$, then $D^\phi$ is the vector valued operator $(D_{X_2}^\phi,\dots, D_{X_m}^\phi)\eqqcolon(D^\phi_2,\dots, D^\phi_m)$.
  The regularity of $\Gamma$ is related to the validity of the equation $D^\phi\phi=\omega$ in an open subset $U\subseteq \sW$, for some $\omega\colon U\to \R^{m-1}$, which can be understood in different ways.  We briefly present some of them here.

 \begin{itemize}
 	\item[] {\em Distributional sense}. Since $\sL$ is one-dimensional, $D^\phi\phi$ is a well-defined distribution, see the last part of \cref{defbroad*}. Thus we could interpret $D^{\phi}\phi=\omega$ in the distributional sense.
 	\item[] {\em Broad* sense}. For every $j=2,\dots, m$ and every point  $a\in U$, there exists a $C^1$ integral curve of $D^\phi_{X_j}$ starting from $a$ for which the Fundamental Theorem of Calculus with derivative $\omega$ holds, see \cref{defbroad*}.
 	\item[] {\em Broad sense}. For every $j=2,\dots, m$ and every point $a\in U$, all the integral curves of $D^\phi_{X_j}$ starting from $a$ are such that the Fundamental Theorem of Calculus with derivative $\omega$ holds, see \cref{defbroad*}.
 	\item[] {\em Approximate sense}. For every $a\in U$, there exist $\delta>0$ and a family $\{\phi_\epsilon\in C^1(B(a,\delta)): \epsilon\in (0,1)\}$ such that $\phi_\epsilon\to \phi$ and $D^{\phi_\epsilon}_{j}\phi_{\epsilon}\to \omega_j$ uniformly on $B(a,\delta)$ as $\epsilon$ goes to zero.
 \end{itemize}
 
 When $\G$ has step 2 and $\sL$ is one-dimensional, the following theorem holds, see \cite[Theorem 6.17]{ADDDLD} for a proof and \cite[Theorem 1.7]{ADDDLD} for an equivalent and coordinate-independent statement. Notice that the statement of the result below needs a choice of coordinates as explained in \cref{sub:step2}, see also \eqref{5.2.0.1}. We also refer the reader to the preliminary section of \cite{ADDDLD} for the notion used in the statement below that are not treated in the current paper.
\begin{theorem}[{\cite[Theorem 6.17]{ADDDLD}}]\label{thm:MainTheoremADDDLD}
	Let $\mathbb{G}$ be a Carnot group of step 2 and rank $m$, and let $\mathbb W$ and $\mathbb L$ be two complementary subgroups of $\mathbb G$, with $\mathbb L$ horizontal and one-dimensional. Let $U\subseteq \mathbb W$ be an open set, and let $\phi\colon U\to \mathbb L$ be a continuous function.
	Then the following conditions are equivalent
	\begin{itemize}
		\item[(a)] $\mathrm{graph}(\phi)$ is a $C^1_{\rm H}$-hypersurface with tangents complemented by $\mathbb L$;
		\item[(b)] $\phi$ is uniformly intrinsically differentiable on $U$;
		\item[(c)] $\phi$ is intrinsically differentiable on $U$ and its intrinsic gradient is continuous;
		\item[(d)] there exists $\omega\in C(U;  \R ^{m-1} )$ such that, for every $a\in U$, there exist $\delta>0$ and a family of functions $\{\phi_\eps\in C^1(B(a,\delta)):\eps\in (0,1)\}$ such that
		\[
		\lim_{\eps\to0}\phi_\eps=\phi, \quad\text{and}\quad\lim_{\eps\to0}D_j^{\phi_\eps}\phi_\eps=\omega _j  \quad\text{in $L^\infty(B(a,\delta))$},
		\]
		for every  $j=2,\dots,m$;	
		\item[(e)]  there exists $\omega\in C(U; \R ^{m-1})$ such that $D^\phi \phi=\omega$ in the broad sense on $U$;
		\item[(f)]  there exists $\omega \in C(U; \R ^{m-1} )$ such that $D^\phi \phi=\omega$ in the broad* sense on $U$.
	\end{itemize}
		Moreover if any of the previous holds, $\omega$ is the intrinsic gradient of $\phi$. 
\end{theorem}

The main result of the current paper is given by the following implication
\begin{equation}\label{eq:implicazionefondamentale}
D^{\phi}\phi=\omega \quad \text{in the sense of distributions}\Rightarrow D^\phi\phi=\omega \quad\text{in the broad* sense},
\end{equation}
in every Carnot group $\G$ of step 2 and for every continuous $\phi\colon U\subseteq \sW\to \sL$, with $U$ open, and $\omega\in C(U;\R^{m-1})$ with $\sL$ one-dimensional, see \cref{thm:Fondamentale1}.
This result allows us to improve \cref{thm:MainTheoremADDDLD} adding a seventh equivalent condition to the  list above\footnote{To complete the chain of implication one also needs (a) $\Rightarrow$ (g) and this follows from \cite[Proposition 4.10]{ADDDLD}.}:
\begin{itemize}
\item[(g)] there exists $\omega \in C(U; \R ^{m-1} )$ such that $D^\phi \phi=\omega$ holds in the distributional sense on $U$.
\end{itemize}
Item (g) allows us to complete the chain of implications of \cref{thm:MainTheoremADDDLD} in the setting of step-2 Carnot groups generalizing the results scattered in \cite{ASCV, BSC, BSCgraphs} where the authors study the same problem in the Heisenberg groups, and \cite{DD19, DD20} where partial generalizations of the results in \cite{ASCV, BSC, BSCgraphs} are obtained in the case of step-2 Carnot groups.

The strategy of the proof of \eqref{eq:implicazionefondamentale} goes as follows. Given a Carnot group $\G$ of step 2, we consider the free Carnot group $\mathbb F$ with step 2 and the same rank of $\G$, see \cref{sub:freestep2} for the precise choice of identifications. We show in \cref{prop:distribuzionaleinGimplicainF} that if $D^\phi \phi=\omega$ in  distributional sense inside $\G$ with some continuous $\omega\in C(U;\R^{m-1})$, then also $D^\psi\psi=\omega \circ \pi$ in distributional sense in $\mathbb F$, where $\psi\coloneqq\pi^{-1}\circ\phi\circ \pi$, and $\pi\colon\mathbb F\to \G$ is the projection. Then, we prove \cref{prop:distrsopraimplicabroadsopra} that tells us that $D^\psi\psi=\omega$ in distributional sense in $\mathbb F$ with $\omega\in C(U;\R^{m-1})$ implies that $D^\psi\psi=\omega$ in the broad* sense, which is exactly implication \eqref{eq:implicazionefondamentale} in the setting of free Carnot groups of step 2. Finally, we prove in \cref{broadinGimplicabroadinF} that $D^\psi\psi=\omega\circ \pi$ in the broad* sense in $\mathbb F$ implies $D^\phi\phi=\omega$ in the broad* sense in $\G$.
The global strategy of lifting the problem to the free Carnot groups resembles the one used in \cite[Section 6]{ADDDLD} and \cite{LDPS19}. 

The main difficulty arises in the proof of \cref{prop:distrsopraimplicabroadsopra} where we have to combine the dimensional reduction given by \cref{lem:dimensionalreduction} and the translation invariance of \cref{prop:translationinvariance} to reduce ourselves to the Burgers' equation of the first Heisenberg group, and then apply the arguments used for this case in \cite[Eqq.\ (3.4) and (3.5)]{Dafermos} and \cite[Step 1, proof of Theorem 1.2]{BSC}. We point out that this argument is essentially different by the one used in \cite{BSC}. One of the reasons for this is that the distributional equation $D^\phi\phi=\omega$ in arbitrary Carnot groups of step 2 has a significantly different structure compared to the one in the Heisenberg groups. For example, consider a Carnot group of dimension 5, step 2 and rank 3 with Lie algebra $\mathfrak g={\rm span}\{X_1,X_2,X_3,X_4, X_5\}$, horizontal layer $V_1\coloneqq{\rm span}\{X_1, X_2, X_3\}$ and where the only nonvanishing commutators are given by $[X_1,X_2]=X_4+X_5$ and $[X_1,X_3]=X_4-X_5$. Define, in exponential coordinates, $\sW:=\{x_1=0\}$ and $\sL:=\{x_2=x_3=x_4=x_5=0\}$. Then, given a continuous $\phi\colon U\subseteq \sW\to \sL$ on an open set $U$, the operators $D^\phi_j\coloneqq D^\phi_{X_j}$ for $j=2,3$ have the following form (see \cite[Example 3.6]{ADDDLD})
\[
\begin{aligned}
D_2^\phi&=\partial_2+\phi\partial_4+\phi\partial_5,\\
D^\phi_3&=\partial_3+\phi\partial_4-\phi\partial_5,
\end{aligned}
\]
which show a nonlinearity in two vertical directions, instead of only one as in the Heisenberg groups.\footnote{Clearly this double nonlinearity can be removed by considering the Lie algebra automorphism such that $\Psi(X_1)=X_1$, $\Psi(X_2)=\frac 12 X_2+\frac 12X_3$, $\Psi(X_3)=\frac 12X_2-\frac 12X_3$. This is basically our idea of properly lifting step-2 Carnot groups to free Carnot groups with the same rank.}


We remark that \cref{prop:distrsopraimplicabroadsopra} and \cref{thm:Fondamentale1} have also an interesting PDE point of view which allows to see the problem independently of the Carnot group structure. Indeed, the result can be read to obtain the following regularity result. Assume that the Burgers' type system $D^\phi\phi=\omega$ holds in the distributional sense for a continuous map $\phi$ and with the continuous datum $\omega$. Then, from each single equation of the system, we infer the following property: for every $j=2,\dots, m$, $\phi$ is (uniformly) Lipschitz continuous on all the integral curves of the operator $D^\phi_j$. In addition, the Fundamental Theorem of Calculus with derivative $\omega$ holds on some particular local family of integral curves of $D^\phi_j$, namely the broad* condition holds, and then also the broad condition holds, see (f)$\Rightarrow$(e) of \cref{thm:Fondamentale1}. Moreover, when we consider all the equations together, we obtain a remarkable piece of information: $\phi$ is $1/2$-little H\"older continuous on the vertical coordinates, see \cref{thm:holder}.

We remark that \cref{thm:MainTheoremADDDLD} complemented with (g) is optimal in step-2 Carnot groups for the following reason. Already in the Engel group, which is the easiest step-3 Carnot group, we can find a continuous map $\phi$ that solves $D^\phi\phi=\omega$ in the sense of distributions for a constant $\omega$ whose graph is not uniformly intrinsically differentiable (UID). We however notice that we do not know at present if implication \eqref{eq:implicazionefondamentale} holds in Carnot groups of higher step, see \cref{rmk:controesempio}.

We briefly describe the situation in which $\omega$ is less regular. In the paper \cite{BCSC}, the authors show that, in Heisenberg groups, $D^\phi\phi=\omega$ holds in the sense of distributions for some $\omega\in L^\infty(U;\R^{m-1})$ if and only if $\phi$ is intrinsically Lipschitz. The validity of \eqref{eq:implicazionefondamentale} with $\omega \in L^\infty(U;\R^{m-1})$ in the setting of step-2 Carnot groups would open to a slightly modified version of \cref{thm:MainTheoremADDDLD} where $\omega\in L^\infty(U;\R^{m-1})$ and (a) is replaced by 
\begin{itemize}
	\item[(a')] {\rm graph}$(\phi)$ is intrinsically Lipschitz for the splitting  given by $\sW$ and $\sL$.
\end{itemize}
Indeed, having $D^\phi\phi=\omega$ in the broad* sense with $\omega\in L^\infty(U;\R^{m-1})$ would imply that $\phi$ is $1/2$-H\"older continuous along vertical directions. This topic is out of the aims of this paper and will be target of future investigations. 

We notice here that if a generalization of the a priori estimate \cite[Lemma 3.1]{MV} would hold in any step-2 Carnot group, then we could improve \cref{thm:MainTheoremADDDLD} replacing (d) with 
\begin{itemize}
\item[(d')] There exists $\omega\in C(U;\R^{m-1})$ and a family of functions $\{\phi_\eps\in C^1(U):\eps\in (0,1)\}$ such that, for every compact set $K\subseteq  U$ and every  $j=2,\dots, m$, one has
\begin{equation*}
\lim_{\eps\to0}\phi_\eps=\phi \quad\text{and}\quad\lim_{\eps\to0}D_j^{\phi_\eps}\phi_\eps=\omega_j \qquad \mbox{in $ L^\infty(K)$}.
\end{equation*}
\end{itemize}
We refer the reader to \cite[Remark 4.14]{ADDDLD} for a discussion of the literature and of the difference between item (d) and item (d'). We also remark that a smooth approximation that does not involve the intrinsic gradient holds in any Carnot group for intrinsic Lipschitz graphs, see \cite[Theorem 1.6]{Vittone20}.

Intrinsic surfaces of higher codimensions have been studied in the Heisenberg groups in \cite{Corni19, CorniMagnani}. For what concerns the approach via distributional solutions, finding a meaning of the distributional system $D^\phi\phi=\omega$ in higher codimension is still open. The main difficulty comes from the fact that it is not known how to give meaning to mixed terms of the form $\phi_i\partial_x \phi_j$. This was already noticed in \cite[Remark 4.3.2]{Koz15}. A weak formulation that goes in this direction is collected in \cite{MST}, where the authors relate zero-level sets of maps in $C^{1,\alpha}_{\rm H}(\mathbb H;\R^2)$ with curves that satisfies certain ``Level Set Differential Equations'', see \cite[Theorem 5.6]{MST}.

\section{Preliminaries}\label{sec:preliminari}	
\subsection{Carnot groups}
We give a very brief introduction on Carnot groups. We refer the reader to e.g.\ \cite{BLU, SC16, LeDonne} for a comprehensive introduction to Carnot groups. A Carnot group $\mathbb G$ is a connected and simply connected Lie group, whose Lie algebra $\mathfrak g$ is stratified. Namely, there exist subspaces $V_1,\dots, V_s$ of the Lie algebra $\mathfrak g$ such that
\[
\mathfrak g = V_1\oplus \dots\oplus V_s,\qquad [V_j,V_1]=V_{j+1} \quad\forall j=1,\dots,s-1,\qquad [V_s,V_1]=\{0\}.
\]
The integer $s$ is called {\em step} of the group $\G$, while $m\coloneqq\dim (V_1)$ is called \emph{rank} of $\G$. We set $n\coloneqq \mathrm{dim}(\mathbb G)$ to be the topological dimension of $\mathbb G$. We equivalently denote by $e$ or $0$ the identity element of the group $\G$. 


Every Carnot group has a one-parameter family of {\em dilations} that we denote by $\{\delta_\lambda: \lambda >0\}$ defined as the unique linear maps on $\mathfrak g$ such that $\delta_\lambda (X)=\lambda^jX$, for every $X\in V_j$. We denote by $\delta_{\lambda}$ both the dilations on $\mathbb G$ and on $\mathfrak g$, with the usual identification given by the exponential map $\exp\colon\mathfrak g\to \G$ which is a diffeomorphism. 
We fix a homogeneous norm $\|\cdot\|$ on $\G$, namely such that $\|\delta_\lambda x\|=\lambda \|x\|$ for every $\lambda>0$ and $x\in \G$, $\|xy\|\leq \|x\|+\|y\|$ for every $x,y\in\mathbb G$, $\|x\|=\|x^{-1}\|$ for every $x\in\mathbb G$, and $\|x\|=0$ if and only if $x=e$. The norm $\|\cdot\|$ induces a left-invariant homogeneous distance and we denote with $B(a,r)$ the open ball of center $a$ and radius $r>0$ according to this distance. We stress that on a Carnot group a homogeneous norm always exists, and every two left-invariant homogeneous distances are bi-Lipschitz equivalent. 

\begin{defi}[Complementary subgroups]\label{def:ComplementarySubgroups}
	Given a Carnot group $\mathbb G$, we say that two subgroups $\mathbb W$ and $\mathbb L$ are \emph{complementary subgroups} in $\G$ if they are {\em homogeneous}, i.e., closed under the action of $\delta_{\lambda}$ for every $\lambda>0$, $\G=\mathbb W\cdot \mathbb L$ and $\mathbb W\cap \mathbb L=\{e\}$. 
\end{defi}

We say that the subgroup $\mathbb L$ is {\em horizontal and $k$-dimensional} if there exist linearly independent $X_1,\dots$, $X_k \in V_1$ such that $\mathbb L=\exp({\rm span}\{X_1,\dots, X_k\})$.
Given two complementary subgroups $\mathbb W$ and $\mathbb L$, we denote the {\em projection maps} from $\G$ onto $\mathbb W$ and onto $\mathbb L$ by $\pi_{\W}$ and $\pi_{\mathbb L}$, respectively. Defining $g_\sW\coloneqq\pi_\sW g$ and $g_\sL \coloneqq \pi_\sL g$ for any $g\in \G$, one has
\begin{equation}\label{eqn:ComponentsSplitting}
g=(\pi_{\W} g)\cdot(\pi_{\mathbb L} g)= g_{\sW}\cdot g_{\sL}.
\end{equation}
\begin{rem} If $\sW$ and $\sL$ are complementary subgroups of $\G$ and $\sL$ is one-dimensional, then it is easy to see that $\sL$ is horizontal. For the sake of clarity, we will always write $\sL$ horizontal and one-dimensional even if one-dimensional is technically sufficient. Notice also that, if $\sW$ and $\sL$ are complementary subgroups and $\sL$ is horizontal, then $\sW$ is a normal subgroup of $\G$.
\end{rem}

\begin{defi}[Adapted basis]
	Denote by $n_j\coloneqq\sum_{i=1}^j{\rm dim}(V_i)$, for $j=1,\dots, s$ and $n_0\coloneqq 0$. We say that a basis $(X_1,\dots, X_n)$ of $\mathfrak g$ is {\em adapted} if the following facts hold
	\begin{itemize}
		\item For every $j=1,\dots, s$, the set $\{X_{n_{j-1}+1},\dots, X_{n_j}\}$ is a basis for $V_j$.
		\item For any $j=1,\dots, s$, the vectors $X_{n_{j-1}+1},\dots, X_{n_j}$ are chosen among the iterated commutators of length $j-1$ of the vectors $X_1,\dots, X_m$.
	\end{itemize}
\end{defi}

\begin{defi}[Exponential coordinates]\label{def:coordinateconletilde}
	Let $\G$ be a Carnot group of dimension $n$ and let $(X_1,\dots,X_n)$ be an adapted basis of its Lie algebra. The {\em exponential coordinates of the first kind} associated with $(X_1,\dots,X_n)$ are given by the one-to-one correspondence
	\[
	\begin{aligned}
	\R^n&\leftrightarrow \G\\
	(x_1,\dots,x_n)&\leftrightarrow\exp\left( x_1X_1+\ldots+x_nX_n\right).
	\end{aligned}
	\]
	It is well known that this defines a diffeomorphism from $\R^n$ to $\G$ that allows us to identify $\G$ with $\R^n$.
\end{defi}

\subsection{Carnot groups of step 2 $\mathbb G$ in exponential coordinates.}\label{sub:step2} We here introduce Carnot groups of step 2 in exponential coordinates. We adopt as a general reference \cite[Chapter~3]{BLU}, but the interested reader could also read the beginning of \cite[Subsection 6.2]{ADDDLD}. {In this subsection $\mathbb G$ will always be an arbitrary Carnot group of step 2.}

We denote with $m$ the rank of $\G$ and we identify $\mathbb G$ with $(\R^{m+h}, \cdot )$ by means of exponential coordinates associated with an adapted basis $(X_1',\dots,X_m',Y_1',\dots,Y_h')$ of the Lie algebra $\mathfrak g$. {In this coordinates, we will identify any point $q\in\mathbb G$ with $q\equiv (x_1,\dots,x_m,y^*_1,\dots,y^*_{h})$}. The group operation $\cdot$ between two elements $q=(x,y^*)$ and $q'=(x',(y^*)')$ is given by
\begin{equation}\label{5.1.0}
q\cdot q'= \left(x+ x',y^*+ (y^*)'-\frac 1 2\langle \mathcal{B} x,  x' \rangle \right),
\end{equation} 
where $\langle \mathcal{B}x,x' \rangle \coloneqq (\langle \mathcal{B}^{(1)}x,x' \rangle, \dots , \langle \mathcal{B}^{(h)} x, x' \rangle)$ and $\mathcal B^{(i)}$ are linearly independent and skew-symmetric matrices in $\R^{m\times m}$, for $i=1,\dots, h$.
For any $i=1,\dots, h$ and any $j,\ell=1,\dots,m$, we set $(\mathcal{B}^{(i)})_{j\ell}\eqqcolon (b_{j\ell}^{(i)})$, and it is standard to observe that we can write
\begin{equation*}
\begin{aligned}
X'_j (p) & = \partial _{x_j}  -\frac{1}{2 } \sum_ {i=1 }^{h} \sum_ {\ell=1 }^{m} b_ {j\ell}^{(i)} x_\ell  \,\partial _{y^*_i},  \quad \mbox{ for } j=1,\dots ,m,\\
Y'_i(p)  & = \partial _{y^*_i }, \,\qquad \qquad \qquad \qquad \qquad \mbox{ for } i=1,\dots , h.
\end{aligned}
\end{equation*}  
 We stress that the operation in \eqref{5.1.0} is precisely the one obtained by means of the Baker-Campbell-Hausdorff formula in exponential coordinates of the first kind associated with the adapted basis $(X'_1,\dots,X'_m,Y'_1,\dots,Y'_h)$. We also stress that
\begin{equation}\label{commutatoripasso2}
[X'_j, X'_\ell]= \sum _{i=1}^h b_{j\ell}^{(i)} Y'_i,  \quad \mbox{and} \quad  [X'_j , Y'_{i}] = 0,\quad \forall j,\ell=1,\dots, m, \;\text{ and }\; \forall i=1,\dots,h,
\end{equation}
so that it is clear that $b^{(i)}_{j\ell}$, with $i=1,\dots,h$, and $1\leq j,\ell\leq m$, are the so-called  {\em structure coefficients}.   

In the sequel we denote by $\W_\G$ and $\mathbb L_\G$ two arbitrary complementary subgroups of $\G$ with $\mathbb L_\G$ horizontal and one-dimensional. Up to choosing a proper adapted basis of the Lie algebra $\mathfrak g$, we may suppose that $\mathbb L_{\mathbb G}=\exp({\rm span}\{X_1\})$. Thus, by means of exponential coordinates we can identify $\mathbb W_{\mathbb G}$ and $\mathbb L_{\mathbb G}$ with $\R^{m+h-1}$ and $\R$, respectively, as follows
\begin{equation}\label{5.2.0.1}
\begin{aligned}
\mathbb L_\G & \equiv\{ (x_1,0\dots , 0) \,:\, x_1\in\R \},\\
\W_\G & \equiv\{ (0,x_2, \dots,x_m,y^*_1,\dots, y^*_h) \,:\, x_i, y^*_k \in\R \mbox{ for }i=2,\dots, m;\, k=1,\dots h \}.
\end{aligned}
\end{equation}

\subsection{Free Carnot groups of step 2 $\mathbb F$ in exponential coordinates.}\label{sub:freestep2} We here introduce free Carnot groups of step 2 in exponential coordinates. We adopt as a general reference \cite[Chapter~3]{BLU}, but the interested reader could also read the beginning of \cite[Subsection 6.1]{ADDDLD}. {In this subsection  $\mathbb F$ will always denote a free Carnot group of step 2 and rank $m$}. Recall that the topological dimension of $\mathbb F$ is $m+ \frac{m(m-1)}{2}$ and denote by $(X_1,\dots, X_m, Y_{21},\dots,Y_{m(m-1)})$ an adapted basis of the Lie algebra of $\mathbb F$ such that $[X_{\ell},X_s]=Y_{\ell s}$ for every $1\leq s<\ell \leq m$. 

If we set $n\coloneqq m+\frac{m(m-1)}{2}$, we can identify $\mathbb F$ with $\R^n$ by means of exponential coordinates associated with the adapted basis $(X_1,\dots, X_m, Y_{21},\dots,Y_{m(m-1)})$. {In this coordinates, we will identify any point $q\in\mathbb F$ with $q\equiv (x_1,\dots,x_m,y_1,\dots,y_{m(m-1)})$}. It is readily seen that, in such coordinates, we have
\begin{equation}\label{eqn:VectorFieldsFree}
\begin{aligned}
X_j &= \partial _{x_j}+ \frac 12 \sum_{ j<\ell\leq m } x_\ell \partial _{y_{\ell j}} -\frac 12 \sum_{ 1\leq \ell<j} x_\ell \partial _{y_{j\ell}}, \quad \mbox{ if } 1\leq j\leq m,\\
Y_{\ell s}&= \partial _{y_{\ell s}}, \hphantom{\frac 12 \sum_{ j<\ell\leq m } x_\ell \partial _{\ell j} -\frac 12 \sum_{ 1\leq \ell<j} x_\ell \partial _{y_{j\ell}}} \qquad  \mbox{ if } 1\leq s<\ell\leq m.
\end{aligned}
\end{equation}
Moreover, for any $q=(x,y)$ and $q'=(x',y')$ in $\mathbb F$, the product $q\cdot q'$  is given by the Baker-Campbell-Hausdorff formula, and yields
\begin{equation*}
\begin{aligned}
(q\cdot q')_j &=x_j+x'_j,\quad \quad \qquad \qquad \qquad \,\,\,\,\, \text{ if } 1\leq j\leq m,\\
(q\cdot q')_{\ell s} &=y_{\ell s}+y'_{\ell s} +\frac 1 2 (x_\ell x'_s-x'_\ell x_s), \quad \text{ if } 1\leq s<\ell\leq m.\\
\end{aligned}
\end{equation*}

\subsection{Projection from $\mathbb F$ to $\G$.} Fix a Carnot group $\mathbb G$  of step 2 and rank $m$ as in \cref{sub:step2} and let $\mathbb F$ be a free Carnot group of rank $m$ and step $s$. By definition of free Carnot groups, there exists a Lie group surjective homomorphism $\pi \colon \mathbb F \to \G$ such that 
\begin{equation}\label{eqn:PISTAR}
\pi _*(X_\ell)=X'_\ell,
\end{equation} 
for any $\ell=1,\dots ,m$ (see e.g.\ \cite[Section~6]{LDPS19}). {We identify $\mathbb F$ and $\mathbb G$ with $\mathbb R^n$ and $\mathbb R^{m+h}$, respectively, by means of exponential coordinates as explained above and in \cref{sub:step2} and \cref{sub:freestep2}}. From \eqref{eqn:PISTAR}, jointly with the very definition of exponential coordinates, we notice that for any $(x,y)\in \R^{n}$, where $x\in \R^m$ and $y\in \R^{m(m-1)/2}$, there exists $y^*\in \R^h$ such that
\begin{equation}\label{proiezione}
\pi(x,y)= (x,y^*).
\end{equation}
Since $\pi$ is a Lie group homomorphism, its differential is a Lie algebra homomorphism. Hence, for any $1\leq s<\ell\leq m$, we have that
\begin{equation*}
\pi_* (Y_{\ell s})=\pi_*  ([X_\ell, X_s])=[\pi_* (X_\ell),\pi_*  (X_s)]=[X'_\ell, X'_s] = \sum _{i=1}^h b_{\ell s}^{(i)} Y'_i,
\end{equation*}
where we used \eqref{commutatoripasso2}, \eqref{eqn:PISTAR}, and the fact that for $1\leq s<\ell\leq m$ one has $[X_\ell , X_s] = Y_{\ell s}$. We can therefore write the following formula
\begin{equation}\label{definizioneesplicitaP}
\begin{split}
& \pi(x_1,\dots , x_m, y_{21},\dots , y_{m(m-1)})  =(x_1,\dots , x_m, y^*_1,\dots , y^*_h), \quad \text{where}\\
& y^*_i  =  \sum _{1\leq s<\ell \leq m}  b_{\ell s}^{(i)} y_{\ell s}, \quad \forall i=1,\dots, h.\\
\end{split}
\end{equation}


\begin{rem}[Main identification]\label{rem:IDENTIFICATION}
	Given a Carnot group $\mathbb G$ of step 2 and rank $m$, and a free Carnot group $\mathbb F$ of step 2 and rank $m$ we work in the coordinates of \cref{sub:step2} and \cref{sub:freestep2}. Let $\mathbb W_{\mathbb G}$ and $\mathbb L_{\mathbb G}$ be two complementary subgroups of $\mathbb G$, with $\mathbb L_{\mathbb G}$ one-dimensional. Up to a proper choice of an adapted basis, we can assume we are working in a basis in which \eqref{5.2.0.1} holds. Thus, taking into account \eqref{proiezione}, we are in a position to lift $\mathbb W_{\mathbb G}$ and $\mathbb L_{\mathbb G}$ to two complementary subgroups $\W_{ \mathbb{F}}$ and $\mathbb{L}_{ \mathbb{F}}$ of $\mathbb F$ such that $\pi_{|_{ \mathbb L _{ \mathbb{F}} }} \colon\mathbb L _{ \mathbb{F}}  \to \mathbb L_\G$ is an isomorphism and $\pi_{|_{ \mathbb W_{ \mathbb{F}} }} \colon\mathbb W_{ \mathbb{F}}  \to \mathbb W_\G$ is onto. In this way we have the following identification
\begin{equation}\label{5.2.0}
\begin{aligned}
\mathbb L_{\mathbb F} & \equiv \left\{ (x_1,\dots,x_m,y_{21},\dots,y_{m(m-1)})\in \R^n : x_2=\dots=x_m=y_{21}=\dots=y_{m(m-1)}= 0 \right\},\\
\W_{\mathbb F} & \equiv \left\{ (x_1,\dots,x_m,y_{21},\dots,y_{m(m-1)})\in\R^n : x_1=0 \right\}.
\end{aligned}
\end{equation}
\end{rem}

\subsection{Projected vector fields in Carnot groups of step 2}\label{sub:projected}
We recall here the definition of projected vector fields \cite[Definition 3.1]{ADDDLD}.
\begin{defi}[Projected vector fields]\label{def:PhiJ}
	Given two complementary subgroups $\mathbb W$ and $\mathbb L$ in a Carnot group $\mathbb G$, and a continuous function $\phi\colon U\subseteq\mathbb W\to\mathbb L$ defined on an open set $U$ of $\mathbb W$, we define, for every $W\in \mathrm{Lie}(\mathbb W)$,  the continuous {\em projected vector field} $D^\phi_W$, by setting 
	\begin{equation}\label{eqn:DefinitionOfDj}
	(D^{\phi}_{W})_{|_w} (f)\coloneqq W_{|_{w\cdot\phi(w)}}( f\circ \pi_{\sW}),
	\end{equation}
	for all $w\in U$ and all $ f\in C^{\infty}(\W)$. When $W$ is an element $X_j$ of an adapted basis $(X_1,\dots,X_n)$ we also write $D^{\phi}_j\coloneqq D^{\phi}_{X_j}$.
\end{defi}

Let us fix $\mathbb G$ a Carnot group of step 2 and rank $m$ along with two complementary subgroups $\mathbb W_{\mathbb G}$ and $\mathbb L_{\mathbb G}$ such that $\mathbb L_{\mathbb G}$ is horizontal and one-dimensional. Assume we have chosen a basis in such a way that \eqref{5.2.0.1} is satisfied. Take $\mathbb F$ the free step-2 Carnot group of rank $m$ and introduce $\mathbb W_{\mathbb F}$ and $\mathbb L_{\mathbb F}$ as in \cref{rem:IDENTIFICATION}. {In this subsection we work in exponential coordinates and we use the identifications and the coordinate representations discussed in \cref{sub:step2}, \cref{sub:freestep2}, and \cref{rem:IDENTIFICATION}}.  We recall that from \cite[Example 3.6 \& Remark 6.9]{ADDDLD} the projected vector fields relative to a continuous function $\phi\colon U\subseteq \sW_\G\to\sL_\G$, with $U$ open, are given by
\begin{equation}\label{operatoriproiettatiinG}
\begin{aligned}
D_{j}^\phi & = \partial _{x_j}-  \sum_{ i=1 }^{h  } \left(  b_{j1}^{(i)} \phi  +\frac 1 2 \sum_{ k=2 }^{m  } x_k b_{jk}^{(i)}\right) \partial _{y^*_i} =  {X'_j}_{\vert U}-  \sum_{i=1}^hb_{j1}^{(i)}  \phi  {Y'_{i}}_{\vert U},  \quad \mbox{ for } j=2,\dots, m,\\
D^\phi_{i} & = \partial_{y^*_i}  =  {Y'_{i}}_{\vert U},  \qquad \mbox{ for }  i=1,\dots, h.
\end{aligned}
\end{equation}

		
    In addition, in the particular case of the free Carnot group $\mathbb F$, given $V\subseteq \mathbb W_{\mathbb F}$ an open set, and given a continuous map $\psi\colon V\subseteq\sW_{\mathbb F}\to \sL_{\mathbb F}$, the projected vector fields are given by
	\begin{equation}\label{operatoriproiettatiinF}
	\begin{aligned}
	D_j^\psi = \partial _{x_j}-\psi \partial _{y_{j1}} + \frac 1 2 \sum_{  j<\ell\leq m } x_\ell\partial _{y_{\ell j}} -  \frac 1 2 \sum_{1<s <j} x_s \partial _{y_{js}}=  {X_j}_{| V}-\psi  {Y_{j1}}_{|V},  \quad &\mbox{ for } j=2,\dots, m,\\
	D^\psi_{\ell s}  = \partial _{y_{\ell s}}  =  {Y_{\ell s}}_{\vert V},  \ \hphantom{=\frac 1 2 \sum_{  j<\ell\leq m } x_\ell\partial _{\ell j} -  \frac 1 2 \sum_{1<s <j} x_s \partial _{js}=  {X_j}_{\vert V}-\psi  {Y_{j1}}_{|V}}\quad &\mbox{ for }  1\leq s<\ell\leq m.
	\end{aligned}
	\end{equation} 
	Then, see also \cite[Remark 6.5]{ADDDLD}, for each $j=2,\dots,m$, every integral curve $\gamma_j\colon I\to \mathbb W_{\mathbb F}\equiv \mathbb R^{n-1}$ of $D^{\psi}_j$ defined on an interval $I\subseteq \R$  has vertical components $y\coloneqq(y_{\ell s})_{1\leq s<\ell\leq m}\colon I\to \mathbb R^{\frac{m(m-1)}2}$  satisfying the following equations
	\begin{equation}\label{eq:ODEverticali}
	\begin{aligned}
	\dot y _{j1} (t)&=-\psi (x_2,\dots ,x_{j-1}, x_j+t,x_{j+1},\dots ,x_m, y(t)), \\
	\dot y _{\ell j} (t)&= \frac 12 x_\ell , \quad \quad  \qquad \mbox{ if }  j<\ell\leq m,\\
	\dot y _{js} (t)&=  -\frac 12 x_s, \quad  \,\qquad \mbox{ if }  1<s <j, \\
	\dot y_{\ell s} (t)&= 0,\quad \qquad  \quad \,\quad  \mbox{ otherwise,} 
	\end{aligned}
	\end{equation}
	where the horizontal components of $\gamma_j(0)$ are $(0,x_2,\dots, x_m)$.

\begin{rem}[Projection on codimension-one subgroups in Carnot groups of step 2]\label{rem:proprietaPq}
	Notice that, if $\sW_{\mathbb G}$ and $\sL_{\mathbb G}$ are complementary subgroups defined as in \eqref{5.2.0.1}, then $\pi_{\sW_{\mathbb G}}\colon \mathbb G\equiv \R^{m+h}\to\mathbb W_{\mathbb G}\equiv \R^{m+h-1}$ is given by
	\begin{equation}\label{eq:proiezionestep2}
	\begin{aligned}
	\pi_{\sW_{\mathbb G}}(x_1,\dots,x_m,y^*_{1}\dots,y^*_{h})=\left(0,x_2,\dots,x_m,\dots, y^*_{i}-\frac12\sum_{j=1}^m b_{j1}^{(i)}x_jx_1, \dots \right), \,\,\, \text{with $i=1,\dots,h$.}
	\end{aligned}
	\end{equation}
	Indeed, it is enough to observe that, thanks to the explicit expression of the product in \eqref{5.1.0}, the following equality holds:
	\[
	\pi_{\sW_{\mathbb G}}(x_1,\dots, x_m, y^*_{1},\dots, y^*_{h})\cdot(x_1,0,\dots,0)=(x_1,\dots, x_m,y^*_{1},\dots, y^*_{h}).
	\]
	For every $q\in \G$, we define the map 
	\begin{equation}\label{eq:definizionePq}
	\begin{aligned}
	P_q\colon \sW_{\mathbb G}&\to \sW_{\mathbb G}\\
	w&\mapsto\pi_{\sW_{\mathbb G}}(q\cdot w).
	\end{aligned}
	\end{equation}
 Set $q=(q_1,\dots,q_m,q_{m+1},\dots,q_{m+h})\in\mathbb G$ and $w=(w_1:=0,w_2,\dots,w_m,w_{m+1},\dots,w_{m+h})\in\mathbb W_{\G}$. By using \eqref{eq:proiezionestep2} and \eqref{5.1.0}, one has, being $i=1,\dots,h$, that the following chain of equalities holds
	\begin{equation}\label{eqn:ExpressionOfPqw}
	\begin{aligned}
	&P_q(w)=\pi_{\sW_{\mathbb G}}(q\cdot w)\\&=\Big(0,q_2+w_2,\dots, q_m+w_m,\dots,q_{m+i}+w_{m+i}+\frac12\sum_{j=1}^m\sum_{\ell=2}^m b^{(i)}_{j\ell}q_jw_\ell-\frac 12\sum_{j=1}^mb^{(i)}_{j1} (q_j+w_j)q_1,\dots\Big)\\&=\Big(0,q_2+w_2,\dots, q_m+w_m,\dots,\\& \hphantom{=}\dots, q_{m+i}+w_{m+i}+\frac12\sum_{j=1}^m\sum_{\ell=2}^mb^{(i)}_{j\ell}q_jw_\ell+\frac 12\sum_{\ell=2}^mb^{(i)}_{1\ell} w_\ell q_1-\frac12\sum_{\ell=2}^mb_{\ell1}^{(i)}q_\ell q_1,\dots\Big)\\&
	=\Big(0,q_2+w_2,\dots, q_m+w_m,\dots,\\& \hphantom{=}\dots, q_{m+i}+w_{m+i}+\frac12\sum_{\ell=2}^mw_\ell\left(\sum_{j=1}^mb^{(i)}_{j\ell}q_j+b^{(i)}_{1\ell} q_1\right)-\frac12\sum_{\ell=2}^mb_{\ell1}^{(i)}q_\ell q_1,\dots\Big),
	\end{aligned}
	\end{equation}
	where we used the fact that the first component of $w$ is zero and that $\mathcal B^{(i)}$ is skew-symmetric and therefore $ b_{11}^{(i)}=0$. If we see $P_q$ as a map from $\R^{m+h-1}$ to $\R^{m+h-1}$, the differential of $P_q$ at a point $w\in \sW$ is identified with a $(m+h-1)\times(m+h-1)$ matrix with the following components
	\begin{equation}\label{eq:componentidifferenziale}
	\begin{aligned}
	&(\de P_q)(w)_{ii}=1, \hspace{0.6\textwidth}\negphantom{$(\de P_q)(w)_{ii}=1$}\forall i=1,\dots,m+h-1,\\
	&(\de P_q)(w)_{m+i-1,\ell-1}=\left(\sum_{j=2}^m\frac12b_{j\ell}^{(i)}q_j+b_{1\ell}^{(i)}q_1\right),\hspace{0.605\textwidth}\negphantom{$(\de P_q)(w)_{m+i-1,\ell-1}=\left(\sum_{j=2}^m\frac12b_{j\ell}^{(i)}q_j+b_{1\ell}^{(i)}q_1\right)$}\forall i=1,\dots,h; \ell=2,\dots,m,\\
	&(\de P_q)(w)_{j\ell}=0,\hspace{0.6\textwidth}\negphantom{$(\de P_q)(w)_{j\ell}=0$} \text{otherwise.}
	\end{aligned}
	\end{equation}
	In particular, $\det (\de P_q)(w)=1$ for any $w\in \sW$.
\end{rem}
	
\subsection{Invariance properties of projected vector fields} We collect here some invariance properties that we will use later on. We introduce the operation of $q$-translation of a function. 
 \begin{defi}[Intrinsic graph of a function]\label{def:IntrinsicGraph}
	Given two complementary subgroups $\mathbb W$ and $\mathbb L$ of a Carnot group $\mathbb G$, and a function $\phi\colon U\subseteq \mathbb W\to\mathbb L$, we define the graph of $\phi$ by setting
	\[
	\mathrm{graph}(\phi)\coloneqq\{\Phi(w):=w\cdot\phi( w): w\in U\}=\Phi(U).
	\]
\end{defi}
\begin{defi}[Intrinsic translation of a function]\label{def:PhiQ}
	Given two complementary subgroups $\mathbb W$ and $\mathbb L$ of a Carnot group $\mathbb G$ and a map ${\phi}\colon{U}\subseteq\mathbb W\to\mathbb L$, we define, for every $q\in\mathbb G$,
	\[
	{U}_q\coloneqq\{a\in\mathbb W: \pi_{\sW}(q^{-1}\cdot a)\in {U}\},
	\]
	and ${\phi}_q\colon{U}_q\subseteq \mathbb W\to\mathbb L$ by setting
	\begin{equation}\label{eqn:Phiq}
	{\phi}_q(a)\coloneqq\left(\pi_{\mathbb L}(q^{-1}\cdot a)\right)^{-1}\cdot {\phi}\left(\pi_{\mathbb W}(q^{-1}\cdot a)\right).
	\end{equation}
\end{defi}
\noindent Notice that $U_q=P_q(U)$, where $P_q$ is defined as in \eqref{eq:definizionePq}. This easily comes from the fact that for every $q\in\mathbb G$ $P_q\circ P_{q^{-1}}=\mathrm{Id}_{\mathbb W}$, see e.g.\ the proof of \cref{prop:translationinvariance}.
 
 The following results can be found in \cite[Proposition 2.10]{ADDDLD} and \cite[Lemma 3.13, and Equations (45)-(46)]{ADDDLD}, respectively.
	\begin{prop}[{\cite[Proposition 2.10]{ADDDLD}}]\label{prop:PropertiesOfIntrinsicTranslation}
		Let $\mathbb W$ and $\mathbb L$ be two complementary subgroups of a Carnot group $\mathbb G$ and let ${\phi}\colon{U}\subseteq\mathbb W\to\mathbb L$ be a function. Then, for every $q\in\mathbb G$, the following facts hold.
		\begin{itemize}
			\item[(a)] ${\rm graph}({\phi}_q)=q\cdot{\rm graph}({\phi})$;
			\item[(b)] $({\phi}_q)_{q^{-1}}={\phi}$;
			\item[(c)] If $\mathbb W$ is normal, then $ U_q=q_{\mathbb W}\cdot\left(q_{\mathbb L}\cdot  U\cdot (q_{\mathbb L})^{-1}\right)$ and
			\begin{equation*}
			{\phi}_q(a)=q_{\mathbb L}\cdot{\phi}((q_{\mathbb L})^{-1}\cdot q_{\mathbb W}^{-1}\cdot a\cdot q_{\mathbb L}),
			\end{equation*}
			for any $a\in U_q$;
			\item[(d)] If $q={\phi}(a)^{-1}\cdot a^{-1}$ for some $a\in {U}$, then 
			\begin{equation*}
			{\phi}_q(e)=e.
			\end{equation*}
		\end{itemize}
	\end{prop}

\begin{lem}[{\cite[Lemma 3.13, and Equations (45)-(46)]{ADDDLD}}]\label{lem:Invariance2IntegralCurve}
	Let $\mathbb W$ and $\mathbb L$ be two complementary subgroups of a Carnot group $\mathbb G$, with $\mathbb L$ $k$-dimensional and horizontal and let ${\phi}\colon{U}\subseteq \mathbb W\to \mathbb L$ be a continuous function defined on $ U$ open. Take $W\in{\rm Lie}(\mathbb W)$, and let us denote $D^{\phi}\coloneqq D^{\phi}_W$. Let $T>0$, $w\in \mathbb W$, and let ${\gamma}\colon [0,T]\to  U$ be a $C^1$-regular solution of the Cauchy problem 
	\begin{equation}\label{eqn:CauchyProblem}
	\begin{cases}
	{\gamma}'(t)=D^{\phi}\circ{\gamma}(t), \\
	{\gamma}(0)= w.
	\end{cases}
	\end{equation}
	Then for every $q\in \G$ there exists a unique $C^1$ map ${\gamma}_q\colon[0,T]\to U_q$ such that
	\begin{equation}\label{eqn:PropertyOfGammaTilde}
	\pi_{\sW}(q^{-1}\cdot {\gamma}_q(t))={\gamma}(t), \qquad \forall t\in[0,T].
	\end{equation}
	In addition, $\gamma_q$ is a solution of the Cauchy problem 
	\begin{equation}\label{eqn:IntegralCurveOfDPhiQ}
	\begin{cases}
	{\gamma}_q'(t)=D^{\phi_q}\circ{\gamma}_q(t), \\
	{\gamma}_q(0)=q_{\mathbb W}\cdot q_{\mathbb L}\cdot  w\cdot (q_{\mathbb L})^{-1}.
	\end{cases}
	\end{equation}
	
	Moreover the following equality holds
	\begin{equation}\label{eqn:Invariance}
	\phi_q(\gamma_q(0))^{-1}\cdot \phi_q(\gamma_q(t)) = \phi(\gamma(0))^{-1}\cdot \phi(\gamma(t)), \qquad \forall t\in[0,T]. 
	\end{equation}
\end{lem}

In the following proposition we prove the invariance of being a distributional solution with respect to $q$-translation. 
\begin{prop}\label{prop:translationinvariance}
	Let $\sW$ and $\sL$ be two complementary subgroups of a step-2 Carnot group $\G$ with $\sL$ one-dimensional. Let $\Omega$ be an open set in $\sW$ and let $\omega \in L^1_{\mathrm{loc}}(\Omega)$. Let us choose coordinates on $\mathbb G$ as explained in \cref{sub:step2}, see also \eqref{5.2.0.1}. Assume that for some $\ell=2,\dots,m$ the map $\phi\colon U\to \sL$ is a distributional solution of the equation $D^\phi_\ell \phi=\omega$ on $U$. Then, for every $q\in \G$, the map $\phi_q$ defined in \cref{def:PhiQ} is a distributional solution of 
	\[
	D^{\phi_q}_\ell\phi_q=\omega\circ P_{q^{-1}},
	\]   
	on the open set $P_q(\Omega)$.
\end{prop}
\begin{proof}
	By item (c) of \cref{prop:PropertiesOfIntrinsicTranslation}, we know that in exponential coordinates $\phi_q(w)=q_1+\phi(P_{q^{-1}}(w))$, where $P_{q^{-1}}$ is defined in \eqref{eq:definizionePq}. Indeed, since $\mathbb W$ is normal, the following equality holds
	\begin{equation}\label{eqn:Comp1}
	P_{q^{-1}}(w)=\pi_\sW(q^{-1}\cdot w)=\pi_\sW((q_{\mathbb L})^{-1}\cdot (q_{\mathbb W})^{-1}\cdot w\cdot q_{\mathbb L}\cdot (q_{\mathbb L})^{-1})=q^{-1}\cdot w \cdot q_{\mathbb L}.
	\end{equation}
	Moreover we claim $P_{q^{-1}}=P_q^{-1}$, for all $q\in\mathbb G$. Indeed, since $\mathbb W$ is normal, the following equality holds
	\begin{equation}\label{eqn:Comp2}
	P_{q}(w)=\pi_\sW(q\cdot w) = \pi_\sW(q\cdot w\cdot q^{-1}\cdot q_{\sW}\cdot q_\sL)=q\cdot w\cdot q^{-1}\cdot q_{\sW}=q\cdot w\cdot (q_\sL)^{-1},
	\end{equation} 
	and hence it is clear by comparing \eqref{eqn:Comp1} and \eqref{eqn:Comp2} that $P_q\circ P_{q^{-1}}=P_{q^{-1}}\circ P_q=\mathrm{Id}_\sW$, for all $q\in\mathbb G$. Moreover notice that from item (c) of \cref{prop:PropertiesOfIntrinsicTranslation} and \eqref{eqn:Comp2} we get that 
	\begin{equation}\label{eqn:PqOmega}
	P_q(\Omega)=\Omega_q, \qquad \text{for all $q\in\mathbb G$}.
	\end{equation} 
	For every $\xi\in C_c^{\infty}(P_q(\Omega))$, using \eqref{operatoriproiettatiinG} we can write the action of the distribution $D_\ell^{\phi_q}\phi_q$ on $\xi$, where we mean that the coordinates are $w=(x_2,\dots,x_m,y_1^*,\dots,y_h^*)\in\mathbb W$, as follows
	\begin{equation}\label{eq:contodistribuzione}
	\begin{aligned}
	\langle D^{\phi_q}_\ell \phi_q, \xi\rangle&=\int_{P_q(\Omega)}\left(-\phi_q\frac{\partial \xi}{\partial x_\ell} +\frac 12\sum_{i=1}^hb_{\ell1}^{(i)}\phi_q^2\frac{\partial \xi}{\partial y^*_{i}}+\frac12\sum_{i=1}^h\sum_{j=2}^mx_jb^{(i)}_{\ell j}\phi_q\frac{\partial \xi}{\partial y^*_{i}}\right)\de w\\&
	=\int_{P_q(\Omega)}\left(-(q_1+\phi\circ P_{q^{-1}})\frac{\partial \xi}{\partial x_\ell} +\frac 12\sum_{i=1}^hb_{\ell1}^{(i)}(q_1+\phi\circ P_{q^{-1}})^2\frac{\partial \xi}{\partial y^*_{i}}\right)\de w\\
	&\hphantom{=}+\int_{P_q(\Omega)}\left(\frac12\sum_{i=1}^h\sum_{j=2}^mx_jb^{(i)}_{\ell j}(q_1+\phi\circ P_{q^{-1}})\frac{\partial \xi}{\partial y^*_{n+i}}\right)\de w\\
	&=\int_{P_q(\Omega)}\left(-\phi\circ P_{q^{-1}}\frac{\partial \xi}{\partial x_\ell} +\sum_{i=1}^hb_{\ell1}^{(i)}\left[\frac 12(\phi\circ P_{q^{-1}})^2+q_1\phi\circ P_{q^{-1}}\right]\frac{\partial \xi}{\partial y^*_{i}}\right)\de w\\
	&\hphantom{=}+\int_{P_q(\Omega)}\left(\frac12\sum_{i=1}^h\sum_{j=2}^mx_jb^{(i)}_{\ell j}(\phi\circ P_{q^{-1}})\frac{\partial \xi}{\partial y^*_{i}}\right)\de w,
	\end{aligned}
	\end{equation}
	where in the third equality we used the fact that $\xi$ has compact support in $P_q(\Omega)$ and $q_1$ does not depend on $w$. Taking into account that, by \cref{rem:proprietaPq}, $\det (\de P_{q})=1$ everywhere on $\mathbb W$, we perform in \eqref{eq:contodistribuzione} the change of variable $w'=P_{q^{-1}}(w)$. Thus, recalling that $P_q\circ P_{q^{-1}} = \mathrm{Id}_{\mathbb W}$, and by exploiting \eqref{eqn:ExpressionOfPqw}, we obtain the following equality
	\begin{equation}\label{eq:contoprimachainrule}
	\begin{aligned}
	\langle D^{\phi_q}_\ell \phi_q, \xi\rangle&=\int_{\Omega}\left(-\phi\frac{\partial \xi}{\partial x_\ell}\circ P_q+\sum_{i=1}^hb_{\ell1}^{(i)}\left[\frac 12\phi^2+q_1\phi\right]\frac{\partial \xi}{\partial y^*_{i}}\circ P_q\right)\de w\\
	&\hphantom{=}+\int_{\Omega}\left(\frac12\sum_{i=1}^h\sum_{j=2}^m(x_j+q_j)b^{(i)}_{\ell j}\phi\frac{\partial \xi}{\partial y^*_{i}}\circ P_q\right)\de w.
	\end{aligned}
	\end{equation}
	We can now use \eqref{eq:componentidifferenziale} to compute the derivatives of $\xi\circ P_q$ as follows
	\begin{equation}\label{eq:chainrule}
	\begin{aligned}
	&\frac{\partial}{\partial x_\ell}(\xi\circ P_q)=\frac{\partial \xi}{\partial x_\ell}\circ P_q+\sum_{i=1}^h\left(\sum_{j=2}^m\frac12 b^{(i)}_{j\ell}q_j+b_{1\ell }^{(i)}q_1\right)\frac{\partial \xi}{\partial y^*_{i}}\circ P_q,\hspace{0.75\textwidth}\negphantom{$\frac{\partial}{\partial w_\ell}(\xi\circ P_q)=\frac{\partial \xi}{\partial w_\ell}\circ P_q+\sum_{i=1}^h\left(\sum_{j=2}^m\frac12 b^{(i)}_{j\ell}q_j+b_{1\ell }^{(i)}q_1\right)\frac{\partial \xi}{\partial w_{n+i}}\circ P_q,$} \forall \ell=2,\dots,m,\\&
	\frac{\partial}{\partial y^*_{i}}(\xi\circ P_q)=\frac{\partial \xi}{\partial y^*_{i}}\circ P_q, \hspace{0.73\textwidth}\negphantom{$\frac{\partial}{\partial w_{n+i}}(\xi\circ P_q)=\frac{\partial \xi}{\partial w_{n+i}}\circ P_q,$}\quad\! \forall i=1,\dots, h.
	\end{aligned}
	\end{equation}
	By using \eqref{eq:chainrule} into \eqref{eq:contoprimachainrule} we get
	\[
	\begin{aligned}
	\langle D^{\phi_q}_\ell \phi_q, \xi\rangle&=\int_\Omega -\phi\left(\frac{\partial (\xi\circ P_q)}{\partial x_\ell}-\sum_{i=1}^h\left(\sum_{j=2}^m\frac12 b^{(i)}_{j\ell}q_j+b_{1\ell }^{(i)}q_1\right)\frac{\partial (\xi\circ P_q)}{\partial y^*_{i}}\right)\de  w\\&\hphantom{=} +\int_\Omega \sum_{i=1}^h\left(b_{\ell1}^{(i)}\left[\frac 12\phi^2+q_1\phi\right]+\frac12 \sum_{j=2}^m(x_j+q_j)b_{\ell j}^{(i)}\phi\right)\frac{\partial (\xi\circ P_q)}{\partial y^*_{i}}\de w\\
	&=\int_\Omega \left(-\phi\frac{\partial (\xi\circ P_q)}{\partial x_\ell}+\frac12\sum_{i=1}^hb_{\ell1}^{(i)}\phi^2\frac{\partial (\xi\circ P_q)}{\partial y^*_{i}}+\frac 12\sum_{i=1}^h\sum_{j=2}^mx_jb_{\ell j}^{(i)}\phi \frac{\partial(\xi\circ P_q)}{\partial y^*_{i}}\right) \de w\\&= \langle D^\phi_\ell \phi, \xi\circ P_q\rangle,
	\end{aligned}
	\]
	where, in order to write the second equality, we used that the matrices $\mathcal{B}$ are skew-symmetric.
	Now, exploiting the last identity, the assumption  and the fact that for every $\xi\in C^{\infty}_c(P_q(\Omega))$ we have $\xi\circ P_q\in C^{\infty}_c(\Omega)$, we conclude
	\[
	\langle D^{\phi_q}_\ell \phi_q, \xi\rangle=\langle D^{\phi}_\ell \phi, \xi\circ P_q\rangle=\int_{\Omega} \omega(\xi\circ P_q)\de w=\int_{P_q(\Omega)} (\omega\circ P_{q^{-1}}) \xi \de w,
	\]
	where in the last equality we changed the variables and exploited the fact that $\det(\de P(q))=1$. 
	By the arbitrariness of $\xi\in C^{\infty}_c(P_q(\omega))$, the proof is complete.
\end{proof}
\subsection{Broad, broad* and distributional solutions of $D^{\phi}\phi=\omega$}
We recall the following definition as a particular case of  \cite[Section 3.4, Definition 3.24]{ADDDLD}. For discussions about the dependence of the definition of broad* regularity on the chosen adapted basis, we refer the reader to \cite[Remark 3.26 \& Remark 4.4]{ADDDLD}.
\begin{defi}[Broad*, broad and distributional solutions]\label{defbroad*}
	Let $\mathbb W$ and $\mathbb L$ be complementary subgroups of a Carnot group $\mathbb G$, with $\mathbb L$ one-dimensional. Let $ U\subseteq \mathbb W$ be open and let $\phi\colon U\to\mathbb L$ be a continuous function. Consider an adapted basis $(X_1,\dots, X_n)$ of the Lie algebra of $\mathbb G$ such that $\mathbb L=\exp({\rm span}\{X_{1}\})$ and $\mathbb W=\exp({\rm span}\{X_{2},\dots,X_n\})$. Let $\omega\coloneqq\left(\omega_{j}\right)_{j=2,\dots,m}\colon U \to \R^{m-1}$  be a continuous function. Up to identifying $\mathbb L$ with $\mathbb R$ by means of exponential coordinates, we say that $\phi \in C(U)$ is a \emph{broad* solution} of $D^\phi \phi=\omega$ in $U$ if for every $a_0\in U$ there exist $0< \delta_2 < \delta_1$ such that $\overline{B(a_0,\delta_1)}\cap\mathbb W\subseteq U$ and there exist $m-1$ maps $E_j^\phi\colon (\overline{B(a_0,\delta_2)}\cap\mathbb W) \times[-\delta _2,\delta _2]\to \overline{B(a_0,\delta_1)}\cap\mathbb W$ for $j=2,\dots, m$, satisfying the following two properties.
	\begin{itemize}
		\item[(a)] For every $a\in \overline{B(a_0, \delta_2)}\cap\mathbb W$ and every $j=2,\dots, m$, the map $E_j^\phi(a)\coloneqq E_j^\phi(a,\cdot)$ is $C^1$-regular and it is a solution of the Cauchy problem
		\[
		\begin{cases}
		\dot \gamma= D^\phi_j\circ \gamma,&\\
		\gamma(0)=a,&
		\end{cases}
		\]
		in the interval $[-\delta_2,\delta_2]$, where the vector field $D^\phi_j\coloneqq D^{\phi}_{X_j}$ is defined in \eqref{eqn:DefinitionOfDj}.
		\item[(b)] For every $a\in \overline{B(a_0,\delta_2)}\cap\mathbb W$, for every $t\in [-\delta_2,\delta_2]$, and every $j=2,\dots,m$ one has
		\[
		\phi(E_j^\phi(a,t))-\phi(a)=\int_0^t \omega_{j} (E_j^\phi(a,s))\,\de s.
		\]
	\end{itemize}
	Moreover, we say that $D^{\phi}\phi=\omega$ in the {\em broad sense on $U$} if for every $W\in {\rm Lie}(\mathbb W)\cap V_1$ and every $\gamma\colon I\to U$ integral curve of $D^{\phi}_W$, it holds that 
	$$
	\frac{\de}{\de s}_{|_{s=t}} (\phi\circ\gamma)(s)=\langle\omega(\gamma(t)),W\rangle , \qquad \forall t\in I,
	$$
	where by $\langle \omega,W\rangle$ is the standard scalar product on $\mathbb R^{m-1}$ in exponential coordinates. 
	
	Finally, let us notice that, for every $j=2,\dots m$, $D^{\phi}_j$ is a continuous vector field with coefficients that might depend polinomially on $\phi$ and on some of the coordinates, see \cref{sub:projected} for the case in which $\mathbb G$ is of step 2 and \cite[Proposition 3.9]{ADDDLD} for the general case. We say that $D^{\phi}\phi=\omega$ holds in the {\em  distributional sense} on $U$ if for every $j=2,\dots,m$ one has $D^{\phi}_j\phi=\omega_j$ in the distributional sense. Notice that the distribution $(D^{\phi}_j)\phi$ is well defined since the coefficients of $D^{\phi}_j$ just contain polynomial terms in $\phi$ and terms depending on the coordinates, see also \cite[Item (a) of Proposition 4.10]{ADDDLD}.
\end{defi}

\section{Reduction of the main theorems to free Carnot groups of step 2}\label{sec:INV}
In this section we analyze the link between distributional and broad* solutions to $D^{\phi}\phi=\omega$ with a continuous datum $\omega$.
We first show that a distributional solution of $D^{\phi}\phi=\omega$ with a continuous datum $\omega$ is a broad* solution inside free Carnot groups $\mathbb F$ of step 2, see \cref{prop:distrsopraimplicabroadsopra} (b). In this proof, a crucial role is played by the particular structure of the projected vector fields inside free Carnot groups of step 2, which produces Burgers' type operators in higher dimensions, see \eqref{operatoriproiettatiinF}. Indeed, combining the invariance result in \cref{prop:translationinvariance} and the dimensional reduction of \cref{lem:dimensionalreduction}, we can reduce ourselves to deal with Burgers' distributional equation with continuous datum on the first Heisenberg group $\HH ^1$, and then exploit the arguments used by Dafermos in \cite{Dafermos} and by Bigolin and Serra Cassano in \cite{BSC}. 

Secondly, by the explicit expression of the projection $\pi$ from $\mathbb F$ to a Carnot group $\G$ of step 2, we prove that being a distributional solution to $D^{\phi}\phi=\omega$ on $\mathbb G$ lifts  to $\mathbb F$, see \cref{prop:distribuzionaleinGimplicainF}. Finally, \cref{broadinGimplicabroadinF} states that the notion of broad* solution is preserved by $\pi$, i.e., a broad* solution on $\mathbb F$ becomes a broad* solution on $\mathbb G$. The resulting strategy resembles the one used in \cite{LDPS19}.
\begin{lem}\label{lem:dimensionalreduction}
	Let $n_1,n_2\in \mathbb N$ and let $\Omega$ be an open set in $ \R^{n_1+n_2+k}$. Let $f_0,f_1,\dots, f_{n_1}\in C(\Omega)$ and assume that, for every $\phi\in C_c^\infty(\Omega)$, one has
	\[
	\iiint \left(\sum_{i=1}^{n_1} f_i(x,y,z)\frac{\partial\phi}{\partial x_i}(x,y,z)+f_0(x,y,z)\phi(x,y,z)\right)\de x\de y\de z=0,
	\]
	where $(x,y,z)\in \R^{n_1}\times \R^{n_2}\times \R^k$.
	Then, for every $z_0\in \R^k$ such that $ \Omega_0\coloneqq \{(x,y)\in \R^{n_1}\times \R^{n_2}: (x,y,z_0)\in \Omega \}$ is nonempty, and any $\phi\in C_c^\infty (\Omega_0)$, one has
	\[
	\iint \left(\sum_{i=1}^{n_1} f_i(x,y,z_0)\frac{\partial\phi}{\partial x_i}(x,y)+f_0(x,y,z_0)\phi(x,y)\right)\de x\de y=0.
	\]
\end{lem}
\begin{proof}
	By translation invariance, we can assume without loss of generality that $z_0=0$. Up to iterate the argument $k$ times, we can also assume without loss of generality that $k=1$. Fix $\widehat\phi\coloneqq\widehat\phi(x,y)$ be such that  ${\rm supt}(\widehat \phi)\subseteq \Omega_0$.
	Choose $\epsilon_0>0$ small enough and consider, for any $\epsilon\in (0,\epsilon_0]$, the map $\phi^\epsilon(x,y,z)\coloneqq\frac{1}{2\epsilon}\phi_0^\epsilon(z)\widehat\phi(x,y)$ with ${\rm supt}(\phi_0^\epsilon)\subseteq [-\epsilon-\epsilon^2,\epsilon+\epsilon^2]$, $\phi_0^\epsilon\geq0$ and $\phi_0^\epsilon=1$ on $[-\epsilon, \epsilon]$, and such that $ {\rm supt}(\widehat \phi)\times{\rm supt}(\phi_0^{\epsilon_0})\subseteq \Omega$. Then, by the hypothesis and Fubini's Theorem we may write
	\begin{equation}\label{eq:fubini}
	\frac{1}{2\varepsilon}\int_{-\epsilon-\epsilon^2}^{\epsilon+\epsilon^2} \phi_0^\varepsilon(z)\left(\iint \sum_{i=1}^{n_1} f_i(x,y,z)\frac{\partial\widehat\phi}{\partial x_i}(x,y)+f_0(x,y,z)\widehat\phi(x,y) \de x\de y\right)\de z=0.
	\end{equation}
	Notice that the function 
	\[
	F(z)\coloneqq\iint \sum_{i=1}^{n_1} f_i(x,y,z)\frac{\partial\widehat\phi}{\partial x_i}(x,y)+f_0(x,y,z)\widehat\phi(x,y) \de x\de y,
	\]
	is continuous on $[-\epsilon_0-\epsilon_0^2,\epsilon_0+\epsilon_0^2]$. We can then decompose the left-hand side of \eqref{eq:fubini} in the following way
	\begin{equation}\label{eq:primomembrospezzato}\begin{aligned}
	\frac{1}{2\varepsilon}\int_{-\epsilon-\epsilon^2}^{\epsilon+\epsilon^2} \phi_0^\varepsilon(z)F(z)\de z=\frac{1}{2\varepsilon}\int_{-\epsilon}^{\epsilon}F(z)\de z+ \frac{1}{2\varepsilon}\int_{[-\epsilon-\epsilon^2,\epsilon+\epsilon^2]\setminus[-\epsilon,\epsilon]}\phi_0^\varepsilon(z)F(z)\de z.
	\end{aligned}
	\end{equation}
	Since $\phi_0^\epsilon\leq 1$, we have 
	\[
	\left|\frac 1{2\epsilon} \int_{[-\epsilon-\epsilon^2,\epsilon+\epsilon^2]\setminus[-\epsilon,\epsilon]}\phi_0^\varepsilon(z)F(z)\de z\right|\leq M \frac{\epsilon^2}{\epsilon}=M \epsilon,
	\]
	where $M$ is the maximum of $|F|$ in $[-\epsilon_0-\epsilon_0^2,\epsilon_0+\epsilon_0^2]$.
	Letting $\epsilon\to 0$ in \eqref{eq:primomembrospezzato}, the thesis follows by means of Lebesgue's Theorem.
\end{proof}

\begin{prop}\label{prop:distrsopraimplicabroadsopra}
Let $\mathbb F$ be a free Carnot group of step 2, rank $m$ and topological dimension $n$, and let $\mathbb W$ and $\mathbb L$ be two complementary subgroups of $\mathbb F$ such that $\mathbb L$ is one-dimensional. Let $\Omega$ be an open subset of $\mathbb W$ and $\psi\colon\Omega\to\mathbb L$ be a continuous function. Choose an adapted basis and exponential coordinates on $\mathbb F$ as in \cref{sub:freestep2}, see also \eqref{5.2.0}. Assume there exists $\omega\in C(\Omega;\R^{m-1})$ such that $D^\psi\psi=\omega$ holds in the distributional sense on $\Omega$. Then, the following facts hold.
\begin{itemize} 
	\item[(a)] For every $j=2,\dots,m$ and for every integral curve $\gamma\colon[0,T]\to \Omega$ of $D^\psi_j$, the map $\psi\circ \gamma\colon [0,T]\to \mathbb L$ is Lipschitz and the Lipschitz constant only depends on $j$ and $\omega$.
	\item[(b)] $D^\psi\psi=\omega$ holds in the broad* sense on $\Omega$.
\end{itemize}
\end{prop}	

\begin{proof}
	{\bf Preliminary dimensional reduction}. Fix $j=2,\dots, m$. Assume $0\in \Omega$, $\psi(0)=0$ and $D^\psi_j\psi=\omega_j$ in the sense of distributions on $\Omega$. Taking \eqref{operatoriproiettatiinF} into account, this amounts to saying that
	\begin{equation}\label{eq:distribuzionale}
	\int \left(-\psi \partial_{x_j}\phi+ \frac{\psi^2}2\partial_{y_{j1}}\phi-\frac12\sum_{j<\ell\leq m} x_\ell\psi\partial_{y_{\ell j}}\phi+\frac 12\sum_{1<s<j} x_s\psi\partial_{y_{js}}\phi\right)\de \mathscr L^{n-1}=\int\phi\omega_j\de\mathscr L^{n-1},
	\end{equation}
	for every $\phi \in C_c^\infty(\Omega)$. 
	Since $\psi$ and $\omega $ are continuous, we are in a position to apply \cref{lem:dimensionalreduction} to the variables $z=(x_2,\dots, x_{j-1}, x_{j+1}, \dots, x_m)$ at $x_2^0=\dots=x_{j-1}^0=x^0_{j+1}=\dots=x_m^0=0$. 
	More precisely $n_1$ is the number of variables that are differentiated, namely $\{x_j, x_{j1}, x_{\ell j}, x_{j s}: j<\ell \leq m, 1<s<j \}$, $k=m-2$ and $n_2$ is the number of remaining variables.
	
	Equation \eqref{eq:distribuzionale} then becomes
	\begin{equation}\label{eq:distribuzionale2}
	\begin{split}
	\int\left( -\psi(0, \dots, 0, x_j,0,\dots, 0, y)\partial_{x_j}\phi(x_j,y)+\frac{\psi^2}{2}(0,\dots,0,x_j,0,\dots, 0, y)\partial_{y_{j1}}\phi(x_j,y)\right)\de x_j \de y\\ =\int\omega_j(0,\dots,0,x_j,0,\dots,0,y)\phi(x_j,y)\de x_j\de y,	
	\end{split}
	\end{equation}
	for any $\phi\in C_c^\infty(\widetilde \Omega)$, where $\widetilde\Omega\coloneqq \{(x_j,y)\in \R\times \R^{m(m-1)/2}: (0,\dots,0, x_j,0\dots, 0, y)\in \Omega\}$. Let us apply again \cref{lem:dimensionalreduction} with $n_1=2$ being the number of variables along which the test $\phi$ is differentiated in \eqref{eq:distribuzionale2}, namely $x_j$ and $y_{j1}$, $n_2=0$, and $k$ being the number of the remaining variables $y_{s\ell}$ with $\ell<s\leq m$ and $(s,\ell)\neq(j,1)$. We apply the Lemma at values $y_{s\ell}^0=0$ and we get
	\begin{equation}\label{eq:distribuzionaleridotta}
	\begin{split}
	\int\left( -\widehat\psi(x_j,y_{j1})\partial_{x_j}\phi(x_j,y_{j1})+\frac{\widehat\psi^2}{2}(x_j,y_{j1}) \partial_{y_{j1}}\phi(x_j,y_{j1})\right)\de x_j\de y_{j1}\\=\int\widehat\omega_j(x_j,y_{j1})\phi(x_j,y_{j1})\de x_j\de y_{j1},
	\end{split}
	\end{equation}
	for every $\phi\in C_c^\infty(\widehat \Omega)$ with $\widehat\Omega\coloneqq\{(x_j,y_{j1})\in \R^2: (0,\dots,0,x_j,0,\dots,0,y_{j1},0,\dots,0)\in \Omega\}$, where 
	\begin{equation}\label{eqn:WhoHat}
	\begin{aligned}
	\widehat \psi(x_j,y_{j1})&\coloneqq\psi(0,\dots,0,x_j,0,\dots,0,y_{j1},0,\dots,0),\\
	\widehat\omega_j(x_j,
	 y_{j1})&\coloneqq\omega_j(0,\dots,0,x_j,0,\dots,0,y_{j1},0,\dots,0).
	 \end{aligned}
	 \end{equation}
	 Observe that the dimensional reduction shown above can be performed in the same way to the distributional equation $ D^{\psi_{q}}\psi_{q}=\omega\circ P_{q^{-1}}$ in the set $P_{q}(\Omega)$, where for any $w\in \Omega$ we have set $q=(w\cdot\psi(w))^{-1}$ and the map $\psi_q$ is defined in \cref{def:PhiQ}. Indeed in this case, if $w\in \Omega$, then by item (c) of \cref{prop:PropertiesOfIntrinsicTranslation} we have $0\in P_q(\Omega)$ and, thanks to item (d) of \cref{prop:PropertiesOfIntrinsicTranslation}, $\psi_q(0)=0$.
	 
	 (a)  It is sufficient to show that, for any $2\leq j\leq m$, there exists a constant $C>0$ such that 
	 \[
	 |\psi(\gamma(T))-\psi(\gamma(0))| \ \leq CT,
	 \]
	 whenever $\gamma\colon[0,T]\to \Omega$ is an integral curve of $D^\psi_j$ with $\gamma(0)=w$, where $w$ is any point in $\Omega$.
	 
	 Fix $2\leq j\leq m$. We first assume $0\in \Omega$ with $\psi(0)=0$ and we consider an integral curve $\gamma\colon[0,T]\to \Omega$ of $D^\psi_j$ such that $\gamma(0)=0$. 
	 Taking \eqref{operatoriproiettatiinF} and \eqref{eq:ODEverticali} into account, we can explicitly write all the components of $\gamma(t)$ as follows:
	  \[
	  \begin{cases}
	  \gamma_j(t)=t,\\
	  \gamma_{j1}(t)=-\int_0^t\widehat \psi(\tau,\gamma_{j1}(\tau))\de\tau,\\
	  \gamma_{i}(t)=0,& \forall i=1,\dots, m, i\neq j,\\
	  \gamma_{\ell s}(t)=0,& \forall (\ell, s) \text{ with $1\leq s<\ell\leq m$ and $(\ell,s)\neq (j,1)$},
	  \end{cases}
	  \]
	  where $\widehat \psi$ is defined in \eqref{eqn:WhoHat}. We can thus define $\widehat \gamma\colon [0,T]\to \R^2$ by letting  $\widehat \gamma(t)\coloneqq (t, \gamma_{j1}(t))$. By the same choice of test functions given in \cite[Eq. (2.5) and (2.6)]{Dafermos}, from \eqref{eq:distribuzionaleridotta} we can derive
	  \begin{equation}\label{eq:stimadafermos}
	  \begin{aligned}
	  \int_{\gamma_{j1}(T)-\epsilon}^{\gamma_{j1}(T)}\widehat \psi(T,x)\de x&-\int_{-\epsilon}^0\widehat\psi(0,x)\de x-\int_0^T\int_{\gamma_{j1}(t)-\epsilon}^{\gamma_{j1}(t)}\widehat \omega_j(t,x)\de x\de t\\&=-\frac 12\int_0^T\left(\widehat \psi(t, \gamma_{j1}(t)-\epsilon)-\widehat \psi (t,\gamma_{j1}(t)) \right)^2\de t,
	  \end{aligned}
	  \end{equation}
	  for every sufficiently small $\epsilon>0$, see \cite[Eq. (3.4)]{Dafermos} with the choice $g=\widehat \omega_j$, $u=\widehat \psi$, $\sigma=0$, $\tau=T$, $\xi=\gamma_{j1}$ and the change of sign of the right hand side with respect to the reference comes by the fact that in our case $f(u)=-\frac 12 u^2$ instead of $f(u)=\frac 12 u^2$. Since the right hand side of \eqref{eq:stimadafermos} is negative, we can write
	  \[
	  \int_{\gamma_{j1}(T)-\epsilon}^{\gamma_{j1}(T)}\widehat \psi(T,x)\de x-\int_{-\epsilon}^0\widehat\psi(0,x)\de x\leq\int_0^T\int_{\gamma_{j1}(t)-\epsilon}^{\gamma_{j1}(t)}\widehat \omega_j(t,x)\de x\de t \leq \epsilon\|\omega_j\|_{L^\infty(\Omega)} T.
	  \]
	  Dividing both sides by $\epsilon$ and letting $\epsilon \to0$, by the continuity of $\psi$ we get
	  \begin{equation}\label{eq:stimaalta}
	  \widehat \psi(T,\gamma_{j1}(T))-\widehat \psi(0,0)\leq \|\omega_j\|_{L^\infty(\Omega)}T.
	  \end{equation}
	  Similarly, by mimicking \cite[Eq. (3.5)]{Dafermos} we can write for every sufficiently small $\epsilon>0$ the equation
	  \[
	  \begin{aligned}
	  \int_{\gamma_{j1}(T)}^{\gamma_{j1}(T)+\epsilon}\widehat \psi(T,x)\de x&-\int_{0}^\epsilon\widehat\psi(0,x)\de x-\int_0^T\int_{\gamma_{j1}(t)}^{\gamma_{j1}(t)+\epsilon}\widehat \omega_j(t,x)\de x\de t\\&=\frac 12\int_0^T\left(\widehat \psi(t, \gamma_{j1}(t)+\epsilon)-\widehat \psi (t,\gamma_{j1}(t)) \right)^2\de t.
	  \end{aligned}
	  \]
	  Noticing that the right hand side is positive one gets
	  \[
	  \int_{\gamma_{j1}(T)}^{\gamma_{j1}(T)+\epsilon}\widehat \psi(T,x)\de x-\int_{0}^\epsilon\widehat\psi(0,x)\de x\geq\int_0^T\int_{\gamma_{j1}(t)}^{\gamma_{j1}(t)+\epsilon}\widehat \omega_j(t,x)\de x\de t \geq -\epsilon\|\omega_j\|_{L^\infty(\Omega)} T.
	  \]
	  Dividing both sides by $\epsilon$ and letting $\epsilon\to 0$ we get
	  \begin{equation}\label{eq:stimabassa}
	  \widehat \psi(T,\gamma_{j1}(T))-\widehat \psi(0,0)\geq -\|\omega_j\|_{L^\infty(\Omega)}T.
	  \end{equation}
	  Combining \eqref{eq:stimaalta} and \eqref{eq:stimabassa} we finally obtain
	  \[
	  |\widehat\psi(T,\gamma_{j1}(T))-\widehat \psi(0,0)|=|\widehat \psi(\widehat \gamma(T))-\widehat \psi(\widehat \gamma(0))|=|\psi\circ \gamma(T)-\psi\circ\gamma(0)|\leq \|\omega_j\|_{L^\infty(\Omega)}T,
	  \]
	  for any integral curve $\gamma\colon[0,T]\to \Omega$ of $D^\psi_j$ with $\gamma(0)=0$.
	  
	  For the general case, assume $w\in\Omega$ and let $\gamma\colon[0,T]\to \Omega$ be an integral curve of $D^\psi_j$ with $\gamma(0)=w$. Setting $q\coloneqq (w\cdot \psi(w))^{-1}$, by \cref{lem:Invariance2IntegralCurve} and in particular \eqref{eqn:Invariance}, there exists an integral curve $\gamma_q\colon[0,T]\to P_{q}(\Omega)$ of $D^{\psi_{q}}_j$ such that $\gamma_q(0)=0$ and
	  \begin{equation}\label{eq:altrainvarianza}
	  \psi_{q}(\gamma_q(t))-\psi_{q}(\gamma_q(0))=\psi(\gamma(t))-\psi(\gamma(0)),\quad \forall t\in [0,T].
	  \end{equation}
	  We also know by \cref{prop:translationinvariance} that $D_j^{\psi_{q}}\psi_{q}=\omega_j\circ P_{q^{-1}}$ in the distributional sense in $P_{q}(\Omega)$. Since $w\in \Omega$, then $0\in P_{q}(\Omega)$ and $\psi_q(0)=0$, see items (c) and (d) of \cref{prop:PropertiesOfIntrinsicTranslation}. We can therefore run the same argument used in the preliminary dimensional reduction and the first part of (a) to $\psi_{q}$, $P_{q}(\Omega)$, $\gamma_q$ and $\omega_j \circ P_{q^{-1}}$, to get that
	  \[
	  |\psi_{q}\circ\gamma_q(T)-\psi_{q}\circ\gamma_q(0)|\leq \|\omega_j\circ P_{q^{-1}}\|_{L^\infty(P_{q}(\Omega))} T.
	  \]
	  The proof of (a) is  complete if we use \eqref{eq:altrainvarianza} and we observe that the Lipschitz constant is uniform by the fact that
	  \[
	  \|\omega_j\|_{L^\infty(\Omega)}=\|\omega_j\circ P_{q^{-1}}\|_{L^\infty(P_{q}(\Omega))}.
	  \]
	  (b) Fix $a_0\in \Omega$, $2\leq j\leq m$ and let $\delta_1>0$ be such that $B(a_0,2\delta_1)\cap \mathbb W\subseteq \Omega$. Up to reducing $\delta_1$, recalling the explicit expression of $P_q$ in \eqref{eq:definizionePq}, we can assume that for every $w\in B(a_0,\delta_1)\cap\mathbb W$ one has $B(0,2\delta_1)\cap\mathbb W\subseteq P_{q}(\Omega)$ where, as before, $q\coloneqq (w\cdot \psi(w))^{-1}$.
	  
	  Let $w\in B(a_0,\delta_1)\cap\mathbb W$. From the fact that $D^{\psi}_j\psi=\omega_j$ in the distributional sense on $\Omega$, we conclude that $D^{\psi_q}_j\psi_q = \omega_j\circ P_{q^{-1}}$ in the distributional sense on $P_q(\Omega)$, where $q:=(w\cdot \psi(w))^{-1}$, see \cref{prop:translationinvariance}. Moreover $0\in P_q(\Omega)$ and $\psi_q(0)=0$, see items (c) and (d) of \cref{prop:PropertiesOfIntrinsicTranslation}. Thus from the preliminary result on the reduction of dimension, see \eqref{eq:distribuzionaleridotta} and \eqref{eqn:WhoHat}, we conclude that $D^{\widehat {\psi_{q}}}_j\widehat {\psi_{q}} = \widehat{\omega_j\circ P_{q^{-1}}} $ holds in the distributional sense on $\widehat {P_q(\Omega)}$. Here we recall that by $D^{\widehat {\psi_{q}}}_j$ we mean the classical Burgers' operator $\partial_j-\widehat {\psi_{q}}\partial_{j1}$ on the open subset $\widehat {P_q(\Omega)}$ of $\mathbb R^2\coloneqq\{(x_j,y_{j1}):x_j,y_{j1}\in\mathbb R\}$.
	  Then we exploit this information and the argument in \cite[Step 1 of proof of Theorem 1.2]{BSC} to find $0<\delta_2<\delta_1$ and a $C^1$-smooth integral curve  $\widehat\gamma\colon[-\delta_2,\delta_2]\to B(0,\delta_1)\cap \widehat {P_q(\Omega)}$ of $D^{\widehat {\psi_{q}}}_j$ such that $ \widehat \gamma(0)=0$ and
	  \begin{equation}\label{eqn:ReducedFTC} \widehat{\psi_{q}}(\widehat\gamma(t))-\widehat{\psi_{q}}(\widehat \gamma(0))=\int_0^t (\widehat{\omega_j\circ P_{q^{-1}}}) \,(\widehat \gamma (s))\,\de s , \quad \forall t\in [-\delta_2,\delta_2].
	  \end{equation}
	  Moreover, by the same argument used in \cite[Step 1 of proof of Theorem 1.2]{BSC}, we can choose $\delta_2\coloneqq\min\{\delta_1/4, \delta_1/(2M_q) \}$, where $M_q\coloneqq\sup_{B(0,2\delta_1)\cap \widehat {P_q(\Omega)}} |\widehat {\psi_{q}}|$. In particular, if $w\in B(a_0,\delta_1)\cap\mathbb W$ and $\delta_1$ is small enough, $M_q$ has a uniform bound depending on the supremum of $\psi$ in some a priori fixed neighborhood of $a_0$, since $\psi_q$ is explicit in terms of $q$, see item (c) of \cref{prop:PropertiesOfIntrinsicTranslation}. As a consequence, up to eventually reducing and fixing $\delta_1$, $\delta_2$ has a positive lower bound independent of $q=(w\cdot\psi(w))^{-1}$, when we allow $w$ to run in $B(a_0,\delta_1)\cap \mathbb W$. We still denote this lower bound with $\delta_2$.
	  
	  Recalling \eqref{eq:ODEverticali} and the first part of this proof, we can write $\widehat \gamma(t)=(t,\gamma_{j1}(t))$ for some $\gamma_{j1}\colon [-\delta_2,\delta_2]\to \R$. For any $w\in B(a_0,\delta_1)\cap\mathbb W$ we can hence define a $w$-dependent $\gamma\colon [-\delta_2,\delta_2]\to B(0,\delta_1)\cap\mathbb W\subseteq P_q(\Omega)$ by letting
	  \begin{equation}\label{eqn:ExprGamma}
	  \gamma(t)\coloneqq (0,\dots, 0, t, 0,\dots,0, \gamma_{j1}(t), 0, \dots, 0),\quad \forall t\in [-\delta_2,\delta_2].
	  \end{equation}
	  Then, since $\gamma(0)=0$, from the particular expression of $D^{\psi_q}_j$, see \eqref{operatoriproiettatiinF} and \eqref{eq:ODEverticali}, and by the fact that $\widehat \gamma$ is an integral curve of $D^{\widehat\psi_q}_j$,
	   we get that $\gamma$ is an integral curve of $D^{\psi_{q}}_j$, and from \eqref{eqn:ExprGamma}, \eqref{eqn:ReducedFTC}, and \eqref{eqn:WhoHat} the following equality holds 
	  \[
	  \psi_{q}(\gamma(t))-\psi_{q}(\gamma(0))=\int_0^t(\omega_j\circ P_{q^{-1}})(\gamma(s)) \,\de s,\quad \forall t\in [-\delta_2,\delta_2].
	  \]
	  Thanks to \cref{lem:Invariance2IntegralCurve} and to item (b) of \cref{prop:PropertiesOfIntrinsicTranslation}, we can translate the integral curve $\gamma$ to an integral curve $\gamma_{q^{-1}}\colon [-\delta_2,\delta_2]\to \Omega$ of $D^{\psi}_j$ with $\gamma_{q^{-1}}(0)=w$, such that, exploiting \eqref{eqn:Invariance}, the following equality holds
	  \[
	  \psi_{q}(\gamma(t))-\psi_{q}(\gamma(0))=\psi(\gamma_{q^{-1}}(t))-\psi(\gamma_{q^{-1}}(0))=\int_0^t \omega_j(\gamma_{q^{-1}}(s))\,\de s,\quad \forall t\in [-\delta_2,\delta_2],
	  \]
	  where the last equality is true since $P_{q^{-1}}\circ\gamma=\gamma_{q^{-1}}$, see \eqref{eqn:PropertyOfGammaTilde} and \eqref{eq:definizionePq}. Thus we have shown that if we fix $a_0\in\Omega$, $j=2,\dots,m$, and $\delta_1$ sufficiently small, we can find $\delta_2$ only depending on $j$, $\omega$, and $\delta_1$ such that for every $w\in B(a_0,\delta_1)$ there exists an integral curve $\gamma\colon[-\delta_2,\delta_2]\to\Omega$ of $D^{\psi}_j$ such that $\gamma(0)=w$ and 
	  \[
	  \psi(\gamma(t))-\psi(\gamma(0))=\int_0^t \omega_j(\gamma(s))\de s, \qquad \forall t\in[-\delta_2,\delta_2].
	  \]
	  By the continuity of $\psi$ this suffices to conclude that $D^{\psi}\psi=\omega$ in the broad* sense on $\Omega$, see \cref{defbroad*}.
\end{proof}

\begin{prop}\label{prop:distribuzionaleinGimplicainF}
Let $\mathbb{G}$ be a Carnot group of step 2, rank $m$ and topological dimension $m+h$, and let $\mathbb W_\G$ and $\mathbb L_\G$ be two complementary subgroups of $\mathbb G$, with $\mathbb L _\G$  one-dimensional. Let $\mathbb F$ be the free Carnot group of step 2, and rank $m$, and choose coordinates on $\mathbb G$ and $\mathbb F$ as explained in \cref{rem:IDENTIFICATION}. Denote with $\mathbb W_{\mathbb F}$ and $\mathbb L_{\mathbb F}$ the complementary subgroups of $\mathbb F$ with $\mathbb L _{\mathbb F}$ one-dimensional such that $\pi(\mathbb W_{\mathbb F})=\mathbb W_\G$ and $\pi(\mathbb L_{\mathbb F})=\mathbb L_\G$, see \eqref{5.2.0}. 
	Let $U$ be an open set in $\sW _\G$ and denote  by ${V} \subseteq\mathbb W_{\mathbb F}$ the open set defined by $  V\coloneqq\pi^{-1}( U)$. Let ${\phi}\colon{U} \to\mathbb L_\G$ be a continuous map and let $ \psi \colon{V} \to\mathbb L_{\mathbb F} $ be the map defined as 
	\[
	{\psi} \coloneqq \pi^{-1} \circ {\phi}  \circ \pi _{|_{ V}}.
	\]
	Assume there exists $\omega\in C(U ; \R^{m-1})$ such that $D^\phi \phi=\omega$ holds in the distributional sense on $U$.  
Then, $\psi$ is a distributional solution to $D^\psi \psi=\omega \circ \pi$ on $V$. 
	
	\end{prop}

\begin{proof}
	 Fix $j=2,\dots, m$. Let us identify any element in $\G$ with $(x,y^*)$ where $x\in \R^m$ and $y^*\in \R^h$, and let us identify any element in $\mathbb F$ with $(x,y)$ where $x\in\mathbb R^m$ and $y\in\mathbb R^{m(m-1)/2}$. Then, taking \eqref{operatoriproiettatiinG} into account, we have that
	\begin{equation}\label{eq:distribuzionaleinG.1}
	\int_{U}\left(-\phi \frac{\partial \xi}{\partial x_j} +\frac 12\sum_{i=1}^hb_{j1}^{(i)}\phi ^2\frac{\partial \xi}{\partial y^*_{i}}+\frac12\sum_{i=1}^h\sum_{\ell =2}^mx_\ell b^{(i)}_{ j \ell}\phi \frac{\partial \xi}{\partial y^*_{i}}\right)\de x\de y^* =\int_U \omega_j \xi \de x\de y^*,
	\end{equation}
	for every $\xi \in C_c^\infty(U)$.  Hence, by exploiting \eqref{operatoriproiettatiinF}, we would like to show that
		\begin{equation}\label{eq:DISTRibuzionaleesatta}
	\begin{aligned}
	\langle D^{\psi}_j \psi, \widetilde \xi\rangle& := \int_{V}\left(-\psi \frac{\partial \widetilde \xi}{\partial x_j} +\frac 12 \psi ^2\frac{\partial \widetilde \xi}{\partial y_{j1}}-\frac12\sum_{j< \ell \leq m} x_\ell \psi\frac{\partial \widetilde \xi}{\partial y_{\ell j }}+ \frac12 \sum_{1<  s <j} x_s \psi\frac{\partial  \widetilde \xi}{\partial y_{ js }} \right)\de x\de y \\&  = \int_{V} (\omega _j \circ \pi) \widetilde \xi \de x\de y,
	\end{aligned}
	\end{equation}
 for every $\widetilde \xi \in C_c^\infty(V)$. We consider the change of variables in $\sW_{\mathbb F}\equiv\R^{n-1}$ given by
 \begin{equation*} \left(x^*_2,\dots, x^*_m, y_1^*, \dots , y^*_h, \hat y^*_{h+1} , \dots, \hat y^*_{ \frac{m(m-1)}{2} } \right)^\top \coloneqq M \left(x_2,\dots, x_m,y_{21},\dots , y_{m(m-1)}\right)^\top, 
 	\end{equation*}
	being $M$ a matrix of order $n-1$ defined as 
 \begin{equation}\label{changematrix}
 	M\coloneqq
	 \begin{pmatrix}
 	I_{m-1}  & 0 & \dots &  0\\
	0&  b_{21}^{(1)} &\dots&  b_{m(m-1)}^{(1)}     \\
	\vdots & \vdots & \ddots & \vdots   \\
	\vdots &  b_{21}^{(h)} &\dots&  b_{m(m-1)}^{(h)} \\
	0 & & \widetilde  M & \\
 	\end{pmatrix},
 	\end{equation}
where $I_{m-1}$ is the identity matrix of order $m-1$ and $\widetilde M $  is a $(\frac{m(m-1)}{2}-h) \times \frac{m(m-1)}{2} $ matrix such that $M$ is invertible. We denote the elements of $\widetilde M$ by $b_{\ell s}^{(i)}$ with $1\leq s<\ell\leq m$ and $i=h+1,\dots,\frac{m(m-1)}{2}$. Such a matrix $\widetilde M$ exists thanks to the fact that the matrices $\mathcal B^{(1)}, \dots ,\mathcal B^{(h)}$ as in \eqref{5.1.0} are linearly independent and then the matrix
 \begin{equation*}
	 \begin{pmatrix}
 	  b_{21}^{(1)} &\dots&  b_{m(m-1)}^{(1)}     \\
	  \vdots & \ddots & \vdots   \\
	 b_{21}^{(h)} &\dots&  b_{m(m-1)}^{(h)} \\
 	\end{pmatrix},
 	\end{equation*}
has maximum rank equal to $h$. Denote for shortness  $x^*\coloneqq(x_2^*,\dots , x^*_m)$, $y^*\coloneqq(y^*_1,\dots , y^*_h)$ and $\hat y^*\coloneqq (\hat y^*_{h+1},\dots , \hat y^*_{ \frac{m(m-1)}{2}})$.
By \eqref{definizioneesplicitaP} we get that, for every $(0,x_2,\dots,x_m,y_{21},\dots,y_{m(m-1)})\in V$, the following quality holds 
	\begin{equation*}
\begin{split}
 \psi(0,x_2,\dots , x_m, y_{21},\dots , y_{m(m-1)}) & = \phi \left(0,x_2,\dots , x_m, \sum _{1\leq s<\ell \leq m}  b_{\ell s}^{(1)} y_{\ell s},\dots , \sum _{1\leq s<\ell \leq m}  b_{\ell s}^{(h)} y_{\ell s} \right)\\
& = \phi \left(0,x^*, y^* \right).
\end{split}
\end{equation*}
Given $\widetilde \xi \in C^{\infty}_c(V)$, we define $\xi(x^*,y^*,\hat y^*)\coloneqq \widetilde \xi\circ M^{-1}(x^*,y^*,\hat y^*)^\top \in C^\infty _c (M(V))$, and then, by using the chain rule in order to write the partial derivatives of $\widetilde \xi $ with respect to $\xi$, we can write
	\begin{equation}\label{eq:DISTR1.1}
	\begin{aligned}
	\langle D^{\psi}_j \psi, \widetilde \xi\rangle& = \int \de \hat y^* \int \Biggl(-\phi \frac{\partial \xi}{\partial x^*_j} +\frac 12 \phi ^2 \sum_{i=1}^hb_{j1}^{(i)} \frac{\partial \xi}{\partial y^*_{i}}-\frac12\sum_{j< \ell \leq m} \sum_{i=1}^hb_{\ell j}^{(i)} x^*_\ell \phi\frac{\partial  \xi}{\partial y^*_{i }}  \\
	  &\hphantom{\sum\sum\sum\sum\sum\sum\sum\sum} +\frac12 \sum_{1<  s <j} \sum_{i=1}^hb_{js}^{(i)} x^*_s \phi\frac{\partial \xi}{\partial y^*_{i }} \Biggl) \frac{1}{|\mbox{det}(M)|} \de x^*\de y^* \\
	+&   \int   \sum_{i=h+1}^{\frac{m(m-1)}{2}} \left( \frac 12 \phi ^2 b_{j1}^{(i)}-\frac12\sum_{j< \ell \leq m}b_{\ell j}^{(i)} x^*_\ell \phi   + \frac12 \sum_{1<  s <j} b_{js}^{(i)} x^*_s \phi \right)  \frac{\partial \xi}{\partial \hat y^*_{i}}\frac{1}{|\mbox{det}(M)|} \de x^*\de y^* \de \hat y^*  \\
	 =&  \int \de \hat y^* \int \left(-\phi \frac{\partial \xi}{\partial x^*_j} +\frac 12 \phi ^2 \sum_{i=1}^hb_{j1}^{(i)} \frac{\partial \xi}{\partial y^*_{i}}+ \frac12\sum_{k= 2}^m \sum_{i=1}^hb_{ jk}^{(i)} x^*_k\phi\frac{\partial  \xi}{\partial y^*_{i }} \right) \frac{1}{|\mbox{det}(M)|}\de x^*\de y^* , \\
	\end{aligned}
	\end{equation}
where $1/ |\mbox{det}(M)|$ is the determinant of the change of variables; we stress that in the last equality we used the fact that $\mathcal B^{(i)}$ is a skew-symmetric matrix  for every $i=1,\dots, h$, and
\begin{equation*}
\begin{aligned}
& \int   \sum_{i=h+1}^{\frac{m(m-1)}{2}} \left( \frac 12 \phi ^2 b_{j1}^{(i)}-\frac12\sum_{j< \ell \leq m}b_{\ell j}^{(i)} x^*_\ell \phi   + \frac12 \sum_{1<  s <j} b_{js}^{(i)} x^*_s \phi \right)  \frac{\partial \xi}{\partial \hat y^*_{i}}\frac{1}{|\mbox{det}(M)|} \de x^*\de y^* \de \hat y^*\\
 &  =:  \int   \sum_{i=h+1}^{\frac{m(m-1)}{2}} A_i(x^*,y^*) \frac{\partial \xi}{\partial \hat y^*_{i}}\frac{1}{|\mbox{det}(M)|} \de x^*\de y^* \de \hat y^*  =0,
 \end{aligned}
\end{equation*}
 because $\xi \in C^\infty _c (M(V))$, together with the use of Fubini's Theorem and the fact that 
   the terms of $A_i(x^*,y^*)$ only depend on the variables $x^*$ and $y^*$.
   
Since, by construction, the projection of $M(V)$ onto the variables $(x^*,y^*)$ is precisely $\pi(V)=U$, we get that $\xi(\cdot, \hat y_0^*) \in C^{\infty}_c(U)$ for every $\hat y_0^*\in\mathbb R^{m(m-1)/2-h}$, and thus by using \eqref{eq:distribuzionaleinG.1}, Fubini's Theorem and the fact that $x^*=x$ since the change of variable is the identity on the horizontal layer, we get
\begin{equation*}
	\begin{aligned}
	\langle D^{\psi}_j \psi, \widetilde  \xi\rangle& = \int \de \hat y^* \int \omega_j\ \xi\   \frac{1}{|\mbox{det}(M)|} \de x^*\de y^*  =  \int \omega_j \ \xi \ \frac{1}{|\mbox{det}(M)|} \de x^*\de y^* \de \hat y^*. 
	\end{aligned}
	\end{equation*}
	Finally, if we consider the reversed change of variables  $(x_2,\dots, x_m, y_{21}, \dots, y_{m(m-1)} )^\top = M^{-1} (x^*,y^*,\hat y ^*)^\top $
	where $M^{-1}$ is the inverse matrix of $M$, see \eqref{changematrix}, and recalling \eqref{definizioneesplicitaP}, it follows
	\begin{equation}
	\begin{aligned}
	\langle D^{\psi}_j \psi, \widetilde \xi\rangle& =  \int \omega_j \left(x,\sum _{1\leq s<\ell \leq m}  b_{\ell s}^{(1)} y_{\ell s},\dots , \sum _{1\leq s<\ell \leq m}  b_{\ell s}^{(h)} y_{\ell s} \right) \widetilde \xi (x,y) \frac{|\mbox{det}(M)|}{|\mbox{det}(M)|} \de x\de y. \\& =  \int (\omega_j \circ \pi)(x,y) \widetilde \xi (x,y)\de x\de y, 
	\end{aligned}
	\end{equation}
	for every $\widetilde \xi \in C_c^\infty(V)$, where $|\mbox{det}(M)|$ is the determinant of the change of variables. Hence \eqref{eq:DISTRibuzionaleesatta} holds and the proof is complete.
\end{proof}

\begin{prop}\label{broadinGimplicabroadinF}
	Let $\mathbb{G}$ be a Carnot group of step 2, rank $m$ and topological dimension $m+h$ and let $\mathbb W_\G$ and $\mathbb L_\G$ be two complementary subgroups of $\mathbb G$, with $\mathbb L _\G$  one-dimensional. Let $\mathbb F$ be the free Carnot group of step 2, rank $m$ and topological dimension $n$, and choose coordinates on $\mathbb G$ and $\mathbb F$ as explained in \cref{rem:IDENTIFICATION}. Denote by $\mathbb W_{\mathbb F}$ and $\mathbb L_{\mathbb F}$ the complementary subgroups of $\mathbb F$ with $\mathbb L _{\mathbb F}$  one-dimensional such that $\pi(\mathbb W_{\mathbb F})=\mathbb W_\G$ and $\pi(\mathbb L_{\mathbb F})=\mathbb L_\G$, see \eqref{5.2.0}. 
	Let $U$ be an open set in $\sW _{\mathbb G}$ and denote  with ${V} \subseteq\mathbb W_{\mathbb F}$ the open set defined by $  V\coloneqq\pi^{-1}(U)$. Let ${\phi}\colon{U} \to\mathbb L_{\mathbb G}$ be a continuous map and let $ \psi \colon{V} \to\mathbb L_{\mathbb F} $ be the map defined as 
	\[
	{\psi} \coloneqq \pi^{-1} \circ {\phi}  \circ \pi _{|_{ V}}.
	\]
	Assume there exists $\omega\in C(U ; \R^{m-1})$ such that $D^\psi \psi=\omega\circ \pi$ holds in the broad* sense on $V$.  
	Then, $\phi$ is a broad* solution to $D^\phi \phi=\omega$ on $U$. 
\end{prop}
\begin{proof}
	In order to give the proof of the statement we first show the following intermediate result: for every $j=2,\dots,m$, every point $q\coloneqq(0,x_2,\dots,x_m,y_{21},\dots,y_{m(m-1)})\in V$, and every integral curve $\gamma\colon[0,T]\to V$ of $D^{\psi}_j$ starting from $q$ we have that $\pi\circ\gamma\colon[0,T]\to U$ is an integral curve of $D^{\phi}_j$ starting from $\pi(q)\eqqcolon(0,x_2,\dots,x_m,y^*_1,\dots,y^*_h)$, see \eqref{proiezione}. Moreover we stress that from \eqref{definizioneesplicitaP} we have $y^*_i  =  \sum _{1\leq s<\ell \leq m}  b_{\ell s}^{(i)} y_{\ell s}$, for all $i=1,\dots, h$.
	
	Take an integral curve $\gamma\colon[0,T]\to U$ of $D^{\psi}_j$ starting from $q$. Then, the components of $\gamma$ satisfy the system of ODEs in \eqref{eq:ODEverticali}. From the explicit expression of the projection in \eqref{definizioneesplicitaP}, we can write the components of $\pi\circ\gamma$ as a linear combination of the components of $\gamma$. Then, exploiting the ODEs in \eqref{eq:ODEverticali}, taking the derivatives of those linear expressions, and by using the definition of $\psi$ in terms of $\phi$ in the statement, one simply obtains that $\pi\circ\gamma\colon[0,T]\to U$ is an integral curve of $D^{\phi}_j$ starting from  $\pi(q)$.
	
	In order to conclude, notice that, from the relation between $\psi$ and $\phi$ in the statement, we obtain the following equivalence 
	\begin{equation}\label{eqn:YEAH}
	\phi(\pi\circ\gamma(t))-\phi(\pi\circ\gamma(0))=\int_0^t\omega_j(\pi\circ \gamma(s))\de s \Leftrightarrow \psi(\gamma(t))-\psi(\gamma(0))=\int_0^t(\omega_j\circ\pi)(\gamma(s))\de s,
	\end{equation} 
	for every integral curve $\gamma\colon[0,T]\to V$ of $D^{\psi}_j$, with $j=2,\dots,m$, and every $t\in [0,T]$. Thus, from the previous observation on the projection of the integral curves and the equivalence \eqref{eqn:YEAH}, we get the thesis by taking the definition of broad* solution in \cref{defbroad*} into account.
\end{proof}

\section{Main theorems}
We are ready to prove the main theorem of this paper, by making use of the invariance results proved in \cref{sec:INV}. The following theorem is a converse of \cite[Corollary 6.15]{ADDDLD}.
\begin{theorem}\label{thm:Fondamentale1}
	Let $\mathbb{G}$ be a Carnot group of step 2 and rank $m$, and let $\mathbb W$ and $\mathbb L$ be two complementary subgroups of $\mathbb G$, with $\mathbb L$ horizontal and one-dimensional. Let $U\subseteq \mathbb W$ be an open set, and let $\phi\colon U\to \mathbb L$ be a continuous function. Choose coordinates on $\mathbb G$ as explained in \cref{sub:step2}, see also \eqref{5.2.0.1}. Assume there exists $\omega \in C(U;\mathbb R^{m-1})$ such that $D^{\phi}\phi=\omega$ holds in the distributional sense on $U$. Then $D^{\phi}\phi=\omega$ holds in the broad* sense on $U$.
\end{theorem}
\begin{proof}
	It directly follows by joining together \cref{prop:distribuzionaleinGimplicainF}, \cref{prop:distrsopraimplicabroadsopra}, and \cref{broadinGimplicabroadinF}.
\end{proof}
By making use of the previous theorem and \cite[Theorem 6.17]{ADDDLD} we obtain the following characterization of $C^1_{\rm H}$-hypersurfaces in Carnot groups of step 2. For the notion of intrinsic differentiabilty we refer the reader to \cite[Definition 2.17]{ADDDLD}, while for the notion of intrinsic gradient we refer the reader to \cite[Definition 2.20 \& Remark 2.21]{ADDDLD}. For the definition of $C^1_{\rm H}$-hypersurface we refer the reader to \cite[Definition 1.6]{FSSC03}. For a detailed account on this notion we refer the reader to the introduction of \cite{ADDDLD} and in particular to \cite[Definition 2.27]{ADDDLD} for the definition of co-horizontal $C^1_{\rm H}$-regular surfaces with complemented tangents.
\begin{theorem}\label{thm:FONDAMENTALE}
	Let $\mathbb{G}$ be a Carnot group of step 2 and rank $m$, and let $\mathbb W$ and $\mathbb L$ be two complementary subgroups of $\mathbb G$, with $\mathbb L$ horizontal and one-dimensional. Let $U\subseteq \mathbb W$ be an open set and let $\phi\colon U\to \mathbb L$ be a continuous function. Choose coordinates on $\mathbb G$ as explained in \cref{sub:step2}, see also \eqref{5.2.0.1}.
	 Then the following conditions are equivalent:
	\begin{itemize}
		\item[(a)] $\mathrm{graph}(\phi)$ is a $C^1_{\rm H}$-hypersurface with tangents complemented by $\mathbb L$;
		\item[(b)] $\phi$ is uniformly intrinsically differentiable on $U$;
		\item[(c)] $\phi$ is intrinsically differentiable on $U$ and its intrinsic gradient is continuous;
		\item[(d)] there exists $\omega\in C(U;  \R ^{m-1} )$ such that, for every $a\in U$, there exist $\delta>0$ and a family of functions $\{\phi_\eps\in C^1(B(a,\delta)):\eps\in (0,1)\}$ such that
		\[
		\lim_{\eps\to0}\phi_\eps=\phi \quad\text{and}\quad\lim_{\eps\to0}D_j^{\phi_\eps}\phi_\eps=\omega _j  \quad\text{in $L^\infty(B(a,\delta))$},
		\]
		for every  $j=2,\dots,m$;
		\item[(e)]  there exists $\omega\in C(U; \R ^{m-1})$ such that $D^\phi \phi=\omega$ holds in the broad sense on $U$;
		\item[(f)]  there exists $\omega \in C(U; \R ^{m-1} )$ such that $D^\phi \phi=\omega$ holds in the broad* sense on $U$;
		\item[(g)] there exists $\omega \in C(U; \R ^{m-1} )$ such that $D^\phi \phi=\omega$ holds in the distributional sense on $U$.
	\end{itemize}
Moreover, if any of the previous holds, $\omega$ is the intrinsic gradient of $\phi$. 
\end{theorem}
\begin{proof}
	The equivalence between (a),(b),(c),(d),(e), and (f) follows form \cite[Theorem 6.17]{ADDDLD}. The implication (g)$\Rightarrow$(f) follows from \cref{thm:Fondamentale1}. The implication (b)$\Rightarrow$(g) follows from \cite[Item (c) of Proposition 4.10]{ADDDLD}.
\end{proof}
\begin{rem}[Intrinsic normal and area formula]
	We stress that if any of the hypotheses in \cref{thm:FONDAMENTALE} is satisfied, we can write the intrinsic normal to $\mathrm{graph}(\phi)$ and an area formula for $\mathrm{graph}(\phi)$ explicitely in terms of the intrinsic gradient $\omega$, see \cite[Item (d) of Proposition 4.10 and Remark 4.11]{ADDDLD}.
\end{rem}
\begin{rem}[Approximation of a distributional solutions to $D^\phi\phi=\omega$]\label{rem:ApproxDistributional}
	The approximating sequence in item (d) of \cref{thm:FONDAMENTALE} is a priori dependent on the point $a\in U$ we choose. This is true because in order to obtain \cite[Theorem 6.17]{ADDDLD}, from which \cref{thm:FONDAMENTALE} follows, we use \cite[Item (b) of Proposition 4.10]{ADDDLD}, in which the approximating sequence is constructed in a way that is a priori dependent on the point $a\in U$. Nevertheless the upgrade of such approximation from a local one on balls to an approximation on arbitrary compact sets, with sequences of functions that are not dependent on the compact set itself, is very likely to be true in the setting of Carnot groups of step 2 by exploiting the same technique explained in \cite[Remark 4.14]{ADDDLD} and based on \cite{MV}. Since this topic does not fit in this paper we will not treat it here, and it will subject of further investigations. 
\end{rem}
\begin{rem}[Counterexample to \cref{thm:FONDAMENTALE} on the Engel group]\label{rmk:controesempio}
	Consider the Engel group $\mathbb E$, i.e., the Carnot group whose Lie algebra $\mathfrak e$ admits an adapted basis $(X_1,X_2,X_3,X_4)$ such that
	\[
	\mathfrak e\coloneqq\mathrm{span}\{X_1,X_2\}\oplus\mathrm{span}\{X_3\}\oplus\mathrm{span}\{X_4\},
	\]
	where $[X_1,X_2]=X_3$, and $[X_1,X_3]=X_4$. We identify $\mathbb E$ with $\mathbb R^4$ by means of exponential coordinates, and we define the couple of homogeneous complementary subgroups $\mathbb W\coloneqq\{x_1=0\}$, and $\mathbb L\coloneqq\{x_2=x_3=x_4=0\}$ in such coordinates. Then, by explicit computations that can be found in \cite[Section 4.4.1]{Koz15}, we get that, for a continuous function $\phi\colon U\subseteq \mathbb W\to\mathbb L$, with $U$ open, the projected vector fields on $\mathbb W$ are
	\begin{equation}\label{eqn:ProjectedVectorFieldsEngel}
	D_{X_2}^{\phi}=\partial_{x_2}+\phi\partial_{x_3}+\frac{\phi^2}{2}\partial_{x_4}, \quad D_{X_3}^{\phi}=\partial_{x_3}+\phi\partial_{x_4}, \quad D^{\phi}_{X_4}=\partial_{x_4}.
	\end{equation}
	Thus, if we consider the function $\phi(0,x_2,x_3,x_4)\coloneqq x_4^{1/3}$ on $\mathbb W$, we get that $D^{\phi}_{X_2}\phi = \frac 16\partial_{x_4}(\phi^3)=\frac 16$ in the distributional sense on $\mathbb W$. On the other hand $\phi\colon\mathbb W\to \mathbb L$ is not uniformly intrinsically differentiable, since it is not $1/3$-little H\"older continuous along the coordinate $x_4$, see \cref{big3.3.11}, while for a function to be uniformly intrinsically differentiable this is a necessary condition, see \cite[Example 5.3]{ADDDLD} and \cite[(a)$\Rightarrow$(c) of Theorem 4.17]{ADDDLD}. Then we conclude that the chain of equivalences of \cref{thm:FONDAMENTALE} cannot be extended already in the easiest step-3 Carnot group.
	
	Nevertheless we do not know whether \cref{thm:Fondamentale1} holds in some cases beyond the setting of step-2 Carnot groups. In particular we do not know whether \cref{thm:Fondamentale1} holds in the Engel group with the splitting previously discussed. Interesting develpoments in the direction of studying whether distributional solutions to Burgers' type equations with non-convex fluxes are also broad solutions are given in \cite{ABC} and \cite{ABC2}. 
\end{rem}
We conclude with the following H\"older property that happens to be a consequence of $\phi$ being a distributional solution to $D^{\phi}\phi=\omega$ with a continuous datum $\omega$. For the purpose, we here recall the definition of little H\"older continuity.

\begin{defi}[little H\"older functions, \cite{Lunardi}]\label{big3.3.11}
	Let $U\subseteq\R^n$ be an open set. We denote by $h^\alpha (U;\R^k)$ the set of all  \emph{$\alpha$-little  H\"older continuous} functions of order $0<\alpha<1$, i.e., the set of maps  $\phi \in C(U;\R^k)$ satisfying
	\begin{equation}\label{equationluna}
	\lim_{r\to 0} \left(\sup \Biggl\{ \, \frac{|\phi (b')-\phi (b)|}{|b'-b|^{\alpha }} : \, b,b'\in U \, , \, 0<|b'-b| <r \, \Biggl\}\right)=0.
	\end{equation}
\end{defi}

\begin{theorem}\label{thm:holder}
	Let $\mathbb{G}$ be a Carnot group of step 2 and rank $m$, and let $\mathbb W$ and $\mathbb L$ be two complementary subgroups of $\mathbb G$, with $\mathbb L$ horizontal and one-dimensional. Let $U\subseteq \mathbb W$ be an open set and let $\phi\colon U\to \mathbb L$ be a continuous function. Choose coordinates on $\mathbb G$ as explained in \cref{sub:step2}, see also \eqref{5.2.0.1}. If one of the items of \cref{thm:FONDAMENTALE} holds, then $\phi$ is $1/2$-little H\"older continuous along the vertical coordinates.  
\end{theorem}
\begin{proof}
	It is a consequence of \cref{thm:FONDAMENTALE} and \cite[Remark 3.23 \& Theorem 6.12]{ADDDLD}.
\end{proof}

\bibliographystyle{alpha}
\bibliography{BCSCBib}
\end{document}